\documentclass[11pt]{article}
\title{\LARGE Improving the Sample and Communication Complexity for Decentralized Non-Convex Optimization: A Joint Gradient Estimation and Tracking Approach}
\author{Haoran Sun$^\dag$, Songtao Lu$^\ddag$ and Mingyi Hong$^\dag$\\
$^\dag$ Department of ECE, University of Minnesota Twin Cities, Minneapolis, MN USA\\
$^\ddag$ IBM Research AI, IBM Thomas J. Watson Research Center, Yorktown Heights, NY USA\thanks{This work is completed when S. Lu was at the University of Minnesota. The authors are supported in part by NSF under Grant CMMI-172775, CIF-1910385 and by AFOSR under grant 19RT0424, ARO under grant 73202-CS.}}

\usepackage{booktabs}
\usepackage{hyperref}
\usepackage{url}
\usepackage{color, colortbl}
\usepackage{graphicx}
\usepackage{amsmath,amssymb,amsthm}
\usepackage{algorithm}
\usepackage{algpseudocode}
\usepackage{xcolor}
\usepackage{flushend}
\usepackage{subfigure}
\usepackage[none]{hyphenat}
\usepackage{pifont} 
\newcommand{\cmark}{\ding{51}} 
\newcommand{\xmark}{\ding{55}} 

\usepackage[]{color-edits}
\addauthor{lu}{orange}
\addauthor{mh}{red}
\addauthor{sun}{blue}

\def\bv{\mathbf v}
\def\bx{\mathbf x}

\def\by{\mathbf y}

\def\bw{\mathbf w}

\def\bI{\mathbf I}

\def\bW{\mathbf W}

\def\bOne{\mathbf 1}

\def\leref#1{Lemma~\ref{#1}}
\def\figref#1{Fig.~\ref{#1}}
\newtheorem{lemma}{Lemma}
\newtheorem{theorem}{Theorem}
\newtheorem{Definition}{Definition}

\newtheorem{assumption}{Assumption}
\newtheorem{Corollary}{Corollary}

\newcommand{\T}{\scriptscriptstyle T}

\def\leref#1{Lemma~\ref{#1}}

\def\thref#1{Theorem~\ref{#1}}

\def\figref#1{Figure~\ref{#1}}

\usepackage{amsmath}
\usepackage{accents}

\newcommand{\norm}[1]{\left\lVert#1\right\rVert}
\newcommand{\ubar}[1]{\underaccent{\bar}{#1}}
\def\remark{\addtocounter{remark}{1}\def\@currentlabel{\theremark}%
\emph{Remark~\theremark}. } \makeatother
\newcounter{remark}

\begin{document}
\maketitle
\begin{abstract}
	 Many modern large-scale machine learning problems benefit from decentralized and stochastic optimization. 
	 Recent works have  shown that utilizing both decentralized computing and local stochastic gradient estimates can outperform state-of-the-art centralized algorithms, in applications involving highly non-convex problems, such as training deep neural networks. 	    
	 
	 In this work, we propose a decentralized stochastic algorithm to deal with certain smooth non-convex problems where there are $m$ nodes in the system, and each  node has a large number of samples (denoted as $n$).  
	 Differently from the majority of the existing decentralized learning algorithms for either stochastic or finite-sum problems, our focus is given to {\it both} reducing the total communication rounds among the nodes, while accessing the minimum number of local data samples. In particular, we propose an algorithm named D-GET (decentralized gradient estimation and tracking), which jointly performs decentralized gradient estimation (which estimates the local gradient using a subset of local samples) {\it and} gradient tracking (which tracks the global full gradient using local estimates). 
	 We show  that, to achieve certain $\epsilon$  stationary solution of the deterministic finite sum problem, the proposed algorithm achieves an $\mathcal{O}(mn^{1/2}\epsilon^{-1})$ sample complexity and an $\mathcal{O}(\epsilon^{-1})$ communication complexity. These bounds significantly improve upon the best existing bounds of $\mathcal{O}(mn\epsilon^{-1})$ and $\mathcal{O}(\epsilon^{-1})$, respectively.  
	 Similarly, for online problems, the proposed method achieves an $\mathcal{O}(m \epsilon^{-3/2})$ sample complexity and  an $\mathcal{O}(\epsilon^{-1})$ communication complexity, while the best existing bounds are  $\mathcal{O}(m\epsilon^{-2})$ and $\mathcal{O}(\epsilon^{-2})$, respectively. 
\end{abstract}
\newpage

\section{Introduction}
For modern large-scale information processing problems, performing centralized computation at a single computing node can require a massive amount of computational and memory resources. 
The recent advances of high-performance computing platforms enable us to  utilize distributed resources to significantly improve the computation efficiency \cite{boyd2011distributed}. These techniques now become essential for many large-scale tasks such as training machine learning models. Modern decentralized optimization shows that partitioning the large-scale dataset into multiple computing nodes could significantly reduce the amount of gradient evaluation at each computing node without significant loss of any optimality \cite{lian2017can}. Compared to the typical parameter-server type distributed system with a fusion center, decentralized optimization has its  unique advantages in preserving data privacy, enhancing network robustness, and improving the computation efficiency \cite{lian2017can, nedic2009distributed, chen2012diffusion, yuan2016convergence}. Furthermore, in many emerging applications  such as collaborative filtering \cite{ali2004tivo}, federated learning \cite{konevcny2016federated}, distributed beamforming \cite{beam12} and dictionary learning \cite{chen2014dictionary}, the data is naturally collected in a decentralized setting, and it is not possible to transfer the distributed data to a central location. 
Therefore, decentralized computation has sparked considerable interest in both academia and industry. 

Motivated by these facts, in this paper we consider the following optimization problem, 
\begin{align}\label{P1}
\min_{\bw \in \mathbb{R}^{d}} f(\bw)=\frac{1}{m}  \sum_{i=1}^m f^i({\bw}),
\end{align}
where  $f^i(\bw): \mathbb{R}^{d} \rightarrow \mathbb{R}$ denotes the loss function which is smooth (possibly non-convex), and $m$ is the total number of such functions. 
We consider the scenario where each node $i \in [m]:= \{1,\cdots, m\}$ can only access its local function $f^i(\bw)$, and can communicate with its neighbors via an undirected and unweighted graph $\mathcal{G}=\{\mathcal{E}, \mathcal{V}\}$. In this work, we consider two typical representations of the local cost functions:
\begin{enumerate}
	\item {\bf Finite-Sum Setting:} Each $f^i({\bw})$ is defined as the average cost of $n$ local samples, that is:
	\begin{align}\label{eq:finite:sum}
	f^i(\bw)= \frac{1}{n} \sum_{j= 1}^{n} f^i_j({\bw}), \forall i
	\end{align}
	where $n$ is the total number of local samples at node $i$, $f^i_j({\bw})$ denotes the cost for $j$th data sample at $i$th node.  
	\item {\bf Online Setting:}  Each $f^i({\bw})$ is defined as the following expected cost 
	\begin{align}\label{eq:online}
	f^i(\bw)= \mathbb{E}_{\xi\sim\mathcal{D}_i}[f^i_{\xi}(\bw)],\forall i
	\end{align}
	where $\mathcal{D}_i$ denotes the data distribution at node $i$. 
\end{enumerate}

To explicitly model the communication pattern, it is conventional to reformulate problem \eqref{P1} as the following consensus problem, by introducing $m$ local variables $\bx_i \in \mathbb{R}^{d}, \; i=1,\cdots, m$,  and use the long vector $\bx$ to stacks all the local variables: $\bx := [\bx_1; \bx_2; \cdots; \bx_m] \in \mathbb{R}^{md}$:
\begin{align}\label{P2}
\min_{\bx \in \mathbb{R}^{md}} f(\bx)=\frac{1}{m}  \sum_{i=1}^m f^i({\bx_i}),
\quad\textrm{s.t.}\quad  \bx_i=\bx_k, \quad\forall (i, k)\in \mathcal{E}.
\end{align}
This way, the loss functions $f^i(\cdot)$'s become separable. 

For the above decentralized non-convex problem \eqref{P2}, one essential task is to find an $\epsilon$ stationary solution $\bx^*:= [\bx^*_1; \cdots; \bx^*_m]\in \mathbb{R}^{md}$ such that
\begin{align}\label{eq:stationarity}
h(\bx^*) =  \left \|\frac{1}{m} \sum_{i=1}^m \nabla f^i(\bx^*_i) \right\|^2 +     {\frac{1}{m}} \sum_i^m \left\|\bx^*_i- \frac{1}{m}\sum_i^m \bx^*_i \right\|^2 \le \epsilon.
\end{align}
Note that the above solution quality measure encodes both the size of local gradient error for classical centralized non-convex problems and the consensus error for decentralized optimization. It is easy to verify that when $h(\bx^*)$ goes to zero, an $\epsilon$ stationary solution for problem \eqref{P1} is obtained.

Many modern decentralized methods can be applied to obtain the above mentioned $\epsilon$ stationary solution for problem \eqref{P2}. In the finite-sum setting  \eqref{eq:finite:sum}, deterministic decentralized methods such as Primal-Dual, NEXT, SONATA, xFILTER \cite{hong2017prox, di2016next, sun2019convergence, sun2018distributed}, which process the local dataset in full batches, typically achieve  $\mathcal{O}(\epsilon^{-1})$ communication complexity (i.e., $\mathcal{O}(\epsilon^{-1})$ rounds of message exchanges are required to obtain $\epsilon$ stationary solution), and  $\mathcal{O}(mn\epsilon^{-1})$ sample complexity (i.e., that many numbers of evaluations of local sample gradients $\{\nabla f^i_j(\bx_i)\}$ are required) \footnote{Note that for the finite sum problem \eqref{eq:finite:sum}, the ``sample complexity" refers to the total number of samples {\it accessed} by the algorithms to compute sample gradient $\nabla f^i_j(\bx_i)$'s. If the same sample $j\in[n_i]$ is accessed $k$ times and each time the evaluated gradients are different, then the sample complexity increases by $k$.}.  Meanwhile,  stochastic methods such as PSGD, D$^2$, stochastic gradient push, GNSD \cite{lian2017can, tang2018d,assran2018stochastic, gnsd19}, which randomly pick subsets of local samples, achieve  $\mathcal{O}(m\epsilon^{-2})$ sample and $\mathcal{O}(\epsilon^{-2})$ communication complexity. 
These complexity bounds indicate that, when the sample size is large (i.e., $\epsilon^{-1} = o(n)$), the stochastic  methods are preferred for lower  sample complexity, but the deterministic methods still achieve lower communication complexity. On the other hand, in the online setting \eqref{eq:online}, only stochastic methods can be applied, and those methods again achieve $\mathcal{O}(m\epsilon^{-2})$ sample and $\mathcal{O}(\epsilon^{-2})$ communication complexity \cite{tang2018d}.

\subsection{Related Works}\label{sec:related}
\subsubsection{Decentralized Optimization}
Decentralized optimization has been extensively studied for convex problems and can be traced back to the 1980s \cite{bertsekas1989parallel}. Many popular algorithms, including decentralized gradient descent (DGD)   \cite{nedic2009distributed, yuan2016convergence}, distributed dual averaging \cite{duchi2011dual}, EXTRA \cite{shi2015extra},  distributed augmented Lagrangian method  \cite{jakovetic2014linear}, adaptive diffusion \cite{chen2012diffusion,lu2014sparsity} and alternating direction method of multipliers (ADMM) \cite{schizas2009distributed, boyd2011distributed, mota2013d, shi2014linear} have been studied in the literature. {We refer the readers to the recent survey \cite{nedic2018network} and the references therein for a complete review.} Recent works also include the study on optimal convergence rates with respect to the network dependency  for strongly convex \cite{scaman2017optimal} and convex \cite{scaman2018optimal} problems.
When the problem becomes non-convex, many algorithms such as primal-dual based methods  \cite{hong2016convergence, hong2017prox}, gradient tracking based methods   \cite{di2016next, daneshmand2016distributed}, and non-convex extensions of DGD methods \cite{zeng2018nonconvex} have been proposed, where the $\mathcal{O}(\epsilon^{-1})$  {iteration and communication complexity} have been shown. Recently, optimal algorithm with respect to the network dependency has also been proposed in \cite{sun2018distributed} with $\mathcal{O}({\gamma^{-1/2}}\times \epsilon^{-1})$ computation  and $\mathcal{O}(\epsilon^{-1})$ communication complexity, where $\gamma$ denotes the spectral gap of the communication graph $\mathcal{G}$. {Note that the above algorithms all require $\mathcal{O}(1)$ full gradient evaluations per iteration, so when directly applied to solve problems where each $f^i(\cdot)$ takes the  form in  \eqref{eq:finite:sum}, they all require {$\mathcal{O}(mn\epsilon^{-1})$} local data samples.}

However, due to the requirement that each iteration of the algorithm needs a full gradient evaluation, the above batch methods can be computationally very demanding. One natural solution is to use the stochastic gradient to approximate the true gradient. Stochastic decentralized non-convex methods can be traced back to \cite{bianchi2013convergence, bianchi2013performance}, and recent advances including DSGD \cite{jiang2017collaborative}, PSGD \cite{lian2017can}, D$^2$ \cite{tang2018d}, GNSD \cite{gnsd19} and stochastic gradient push \cite{assran2018stochastic}. However, the large variance coming from the stochastic gradient estimator and the use of diminishing step size slow down the convergence, resulting at least $\mathcal{O} (m\epsilon^{-2})$ sample and $\mathcal{O} (\epsilon^{-2})$ communication cost. 

Recent works also include studies on developing distributed algorithms having second-order guarantees \cite{hong2018gradient, daneshmand2018second, vlaski2019distributed, swenson2019distributed}. This is an interesting research direction that further showcases the strength of decentralized algorithms. However, to limit the scope of this paper, we only focus on  convergence issues of the decentralized method  to first-order solutions (as defined in \eqref{eq:stationarity}).

\subsubsection{Variance Reduction}
{Consider the following  non-convex finite sum problem: {$\min_{\bw \in \mathbb{R}^{d}} f(\bw)=\frac{1}{mn}\sum_{j=1}^{mn} f_j(\bw)$}.} If we assume that $f(\cdot)$ has Lipschitz gradient, and  directly apply  the vanilla gradient descent (GD) method on $f(\bw)$, then it requires  {$\mathcal{O} (mn\epsilon^{-1})$}  gradient evaluations to reach $\|\nabla f(\bw)\|^2 \le \epsilon$ \cite{nesterov1998introductory}. 
When  {$m\times n$} is large, it is usually preferable to process a subset of data each time. In this case, stochastic gradient descent (SGD) can be used to  
achieve an $\mathcal{O} (\epsilon^{-2})$ convergence rate \cite{ghadimi2013stochastic}.

To bridge the gap between the GD and SGD,  many variance reduced  gradient estimators have been proposed, including  SAGA \cite{defazio2014saga} and SVRG \cite{johnson2013accelerating}. The idea is to reduce the variance of the stochastic gradient estimators and substantially improves the convergence rate. In particular, the above approaches have been shown to achieve sample complexities of  {$\mathcal{O} ((mn)^{2/3}\epsilon^{-1})$}   for finite sum problems \cite{reddi2016stochastic, allen2016variance, lei2017non}  and $\mathcal{O} (\epsilon^{-5/3})$ for online problem \cite{lei2017non}. Recent works further improve the above gradient estimators and achieve  {$\mathcal{O} ((mn)^{1/2}\epsilon^{-1})$} sample complexity  for finite sum problems \cite{nguyen2019optimal, fang2018spider, wang2018spiderboost, zhou2018stochastic} and $\mathcal{O} (\epsilon^{-3/2})$ sample complexity for online problems \cite{fang2018spider, wang2018spiderboost}. At the same time, the  {$\mathcal{O} ({(mn)}^{1/2}\epsilon^{-1})$} sample complexity is shown to be optimal when  {$m\times n\le \mathcal{O} (\epsilon^{-2})$} \cite{fang2018spider}.
However, it is important to mention that one has to be careful when comparing various complexity bounds. This is because, one key assumption that enables the variance reduced algorithms to achieve improved complexity with respect to $m\times n$ is that, {\it each} component function $f_j(\cdot)$ has Lipschitz gradient (therefore they are ``similar" in certain sense), while the vanilla GD only requires that the {\it sum} of the component functions has Lipschitz gradient.

\subsubsection{Decentralized Variance Reduction}
The variance reduced decentralized optimization has been extensively studied for convex problems.  The DSA proposed in \cite{mokhtari2016dsa} combines the algorithm design ideas from  EXTRA \cite{shi2015extra} and SAGA \cite{defazio2014saga}, and achieves the first expected linear convergence for decentralized stochastic optimization. Recent works also include the DSBA \cite{shen2018towards},  diffusion-AVRG \cite{yuan2018variance}, ADFS \cite{hendrikx2019accelerated}, SAL-Edge \cite{wang2019edge}, GT-SAGA \cite{xin2019variance}, and Network-DANE \cite{li2019communication}.  In particular, the DSBA \cite{shen2018towards} introduces the monotone operator to reduce the dependence on the problem condition number compared to DSA \cite{mokhtari2016dsa}. Diffusion-AVRG combines the exact diffusion \cite{yuan2018exact} with the AVRG \cite{ying2018stochastic}, and extend the results to scenarios that the size of the data is unevenly distributed. ADFS \cite{hendrikx2019accelerated} further uses randomized pairwise communication to achieve optimal network scaling. The work \cite{wang2019edge} combines the augmented Lagrange (AL) based method with SAGA \cite{defazio2014saga} to allow flexible mixing weight selections. GT-SAGA \cite{xin2019variance} improves the joint dependence on the condition number and number of samples per node. The Network-DANE \cite{li2019communication} studies the Newton-type method and establishes the linear convergence for quadratic losses. However, when the problem becomes non-convex, to the best of our knowledge, no algorithms with provable guarantees are available.

\subsection{Our Contribution}

Compared with the majority of the existing decentralized learning algorithms for either stochastic or deterministic problems, the focus of this work is given to {\it both} reducing the total {communication} and sample complexity. Specifically, we propose a decentralized gradient estimation and tracking (D-GET) approach, which uses a subset of samples to estimate the local gradients (by utilizing modern variance reduction techniques \cite{fang2018spider, nguyen2017sarah}), while using the differences of past local gradients to track the global gradients (by leveraging the idea of decentralized gradient tracking \cite{di2016next,pu2018distributed}). 
Remarkably, the proposed approach enjoys a sample complexity of  {$\mathcal{O}(mn^{1/2}\epsilon^{-1})$} and communication complexity of $\mathcal{O}(\epsilon^{-1})$ for finite sum problem \eqref{eq:finite:sum}, which outperforms all existing decentralized methods\footnote{Note that, as mentioned before, deterministic batch gradient based methods such as xFILTER, Prox-PDA, NEXT, EXTRA achieve an $\mathcal{O}(mn\epsilon^{-1})$ sample complexity. However, to be fair, one cannot directly compare those bounds with what can be achieved by sample based, variance reduced methods, since the assumptions on the Lipschitz gradients are slightly different. }.  
The sample complexity rate is $\sqrt{m}$ worse than the known sample complexity lower bound for centralized problem \cite{fang2018spider}, and the communication complexity matches the existing communication lower bound \cite{sun2018distributed} for decentralized non-convex optimization (in terms of the dependency in $\epsilon$). 
Furthermore, the proposed approach is also able to achieve $\mathcal{O}(m \epsilon^{-3/2})$ sample complexity and $\mathcal{O}(\epsilon^{-1})$  communication complexity for the online problem \eqref{eq:online}, reducing the best existing bounds {(such as those obtained in \cite{tang2018d, gnsd19})} by factors of $\mathcal{O}(\epsilon^{-1/2})$ and $\mathcal{O}(\epsilon^{-1})$, respectively.  
We illustrate the main results of this work in  \figref{fig:1}, and compare the gradient and communication cost for state-of-the-art decentralized non-convex optimization approaches in Table   \ref{fig:table_compare}\footnote{For deterministic batch algorithms such as DGD, NEXT, Prox-PDA and xFILTER, the bounds are obtained directly by multiplying their respective convergence rates with $m\times n$, since when directly applied to solve finite-sum problems, each iteration requires $\mathcal{O}(1)$ full gradient evaluation.}.
Note that in Table \ref{fig:table_compare},  by {\it constant stepsize} we mean that stepsize is not dependent on the target accuracy $\epsilon$, {nor is it dependent on the iteration number}.

\begin{figure}[t]
	\centering
	\begin{minipage}[c]{0.48\linewidth}
		\includegraphics[width=\linewidth]{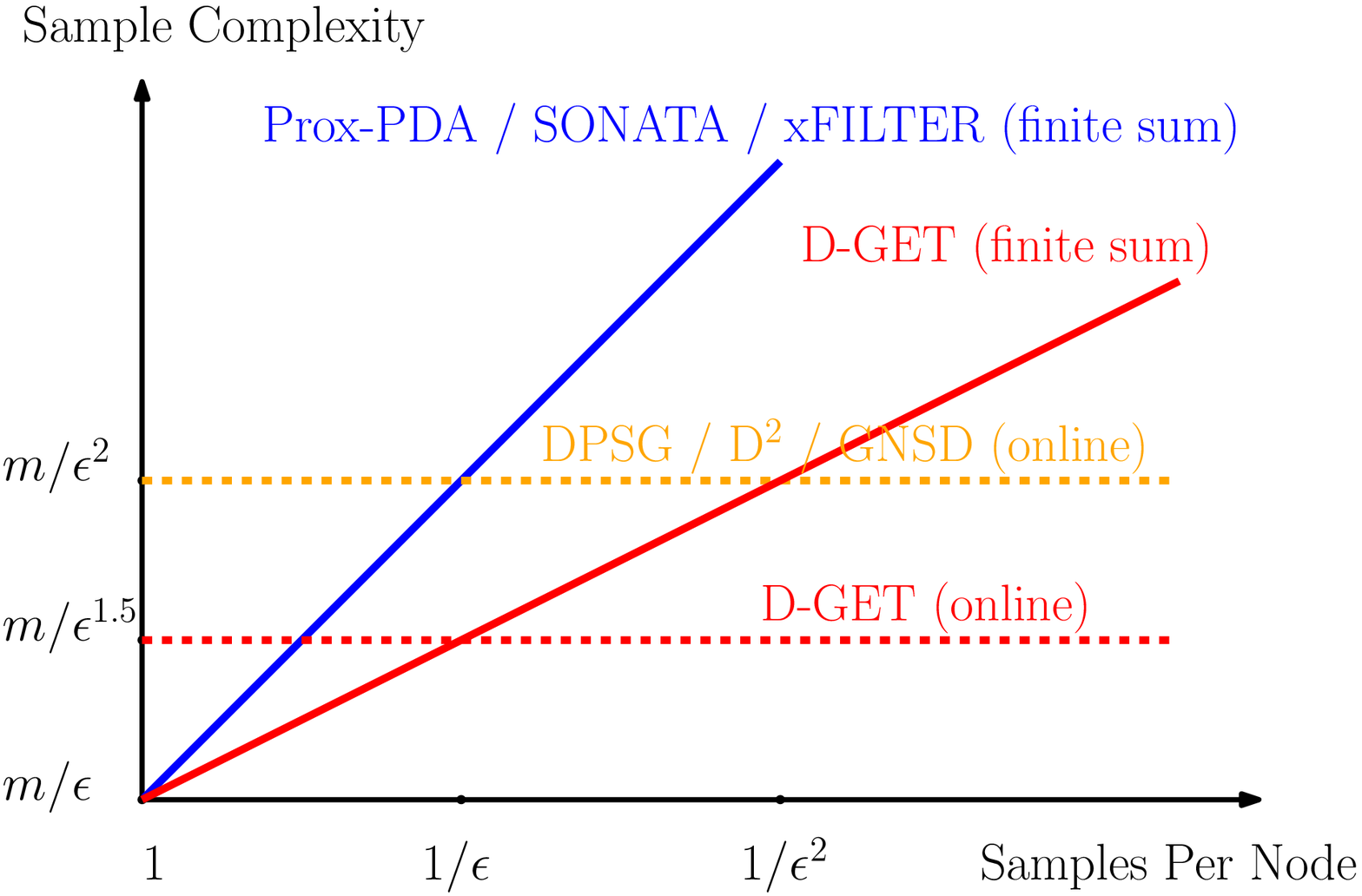}
		\centering{(a) Sample Complexity}
	\end{minipage}
	\begin{minipage}[c]{0.48\linewidth}
		\includegraphics[width=\linewidth]{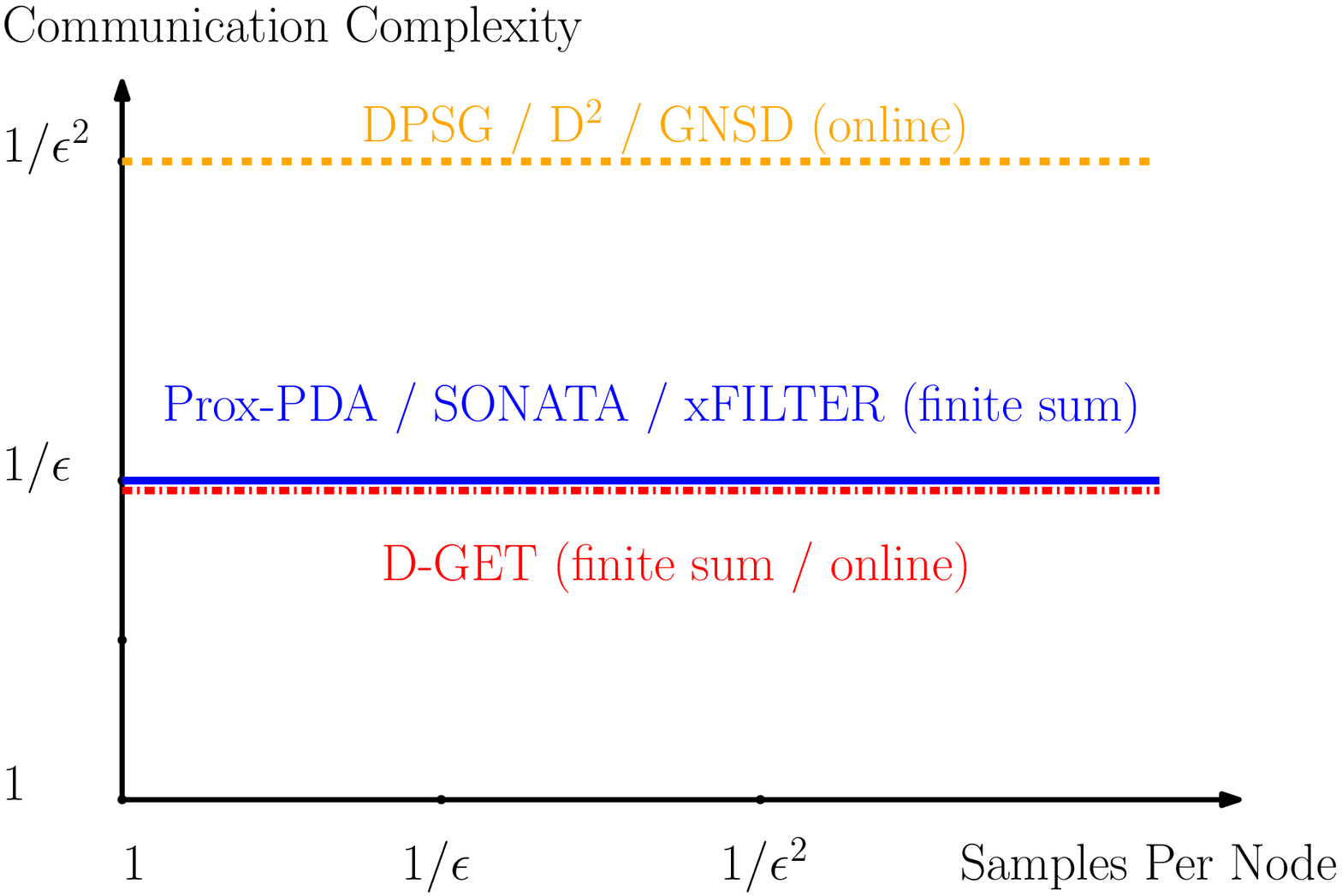}
		\centering{(b) Communication Complexity}
	\end{minipage}
	\caption{\small Comparison of the sample and communication complexities for a number of decentralized methods. Existing deterministic methods enjoy lower sample complexity at smaller sample sizes, but such complexity scales linearly when the number of samples increases. Stochastic methods generally suffer from high communication complexity. The proposed D-GET bridges the gap between existing deterministic and stochastic methods, and achieves the optimal sample and communication complexities. Note that online methods can also be applied for finite sum problems, thus the actual sample complexity of D-GET is the minimum rate of both cases.}
	\label{fig:1}
\end{figure}

\begin{table}[t]
\caption{Comparison of algorithms on decentralized non-convex optimization}
\label{fig:table_compare}
\begin{center}
\begin{small}
\begin{sc}
\begin{tabular}{lccccccc}
\toprule
Algorithm & Constant Stepsize & Finite-Sum & Online & Communication \\
\midrule
DGD \cite{zeng2018nonconvex} & \xmark  & $\mathcal{O}(mn\epsilon^{-2})$ & \xmark & $\mathcal{O}(\epsilon^{-2})$\\
% NEXT \cite{di2016next} & \xmark  & - & \xmark & -\\
SONATA \cite{sun2019convergence} & \cmark  & $\mathcal{O}(mn\epsilon^{-1})$ & \xmark & $\mathcal{O}(\epsilon^{-1})$\\
Prox-PDA \cite{hong2017prox} & \cmark & $\mathcal{O}(mn\epsilon^{-1})$ & \xmark & $\mathcal{O}(\epsilon^{-1})$\\
xFILTER \cite{sun2018distributed} & \cmark  & $\mathcal{O}(mn\epsilon^{-1})$ & \xmark & $\mathcal{O}(\epsilon^{-1})$\\
PSGD \cite{lian2017can} & \xmark & $\mathcal{O}(m\epsilon^{-2})$ & $\mathcal{O}(m\epsilon^{-2})$ & $\mathcal{O}(\epsilon^{-2})$\\
D$^2$ \cite{tang2018d}  & \cmark & $\mathcal{O}(m\epsilon^{-2})$ & $\mathcal{O}(m\epsilon^{-2})$ & $\mathcal{O}(\epsilon^{-2})$\\
GNSD \cite{gnsd19}  & \cmark & $\mathcal{O}(m\epsilon^{-2})$ & $\mathcal{O}(m\epsilon^{-2})$  & $\mathcal{O}(\epsilon^{-2})$\\
{\bf D-GET} (this work) & \cmark & $\mathcal{O}( m\sqrt{n}\epsilon^{-1})$ & $\mathcal{O}( m \epsilon^{-3/2})$ & $\mathcal{O}(\epsilon^{-1})$\\
{\bf Lower Bound} \cite{fang2018spider, sun2018distributed} & - & $\mathcal{O}( \sqrt{mn}\epsilon^{-1})$ & - & $\mathcal{O}(\epsilon^{-1})$\\
\bottomrule
\end{tabular}
\end{sc}
\end{small}
\end{center}
\end{table}

\section{The Finite Sum Setting}\label{Section_FiniteSum}
In this section, we  consider the non-convex decentralized optimization problem \eqref{P2} with finite number of  samples as defined in \eqref{eq:finite:sum}, which is restated below:
\begin{align}\label{P2_finite}
\min_{\bx \in \mathbb{R}^{md}} f(\bx)=\frac{1}{mn}  \sum_{i=1}^m   \sum_{j= 1}^{n} f^i_j({\bx_i}),
\quad\textrm{s.t.}\quad  \bx_i=\bx_k, \quad\forall (i, k)\in \mathcal{E}.\tag{P1}
\end{align}
%\subsection{Assumptions}
We make the following standard assumptions on the above problem:
\begin{assumption}  \label{A1}
The objective function has Lipschitz continuous gradient with constant $L$:
\begin{equation}
\|\nabla f_j^i(\bx_i)-\nabla f_j^i(\bx'_i)\|\le L\|\bx_i-\bx'_i\|,\forall i, j
\end{equation}
which also implies
\begin{subequations}
\begin{align}
&\|\nabla f^i(\bx_i)-\nabla f^i(\bx'_i)\|^2 \le L^2 \|\bx_i-\bx'_i\|^2,\forall i\\
&\|\nabla f_j(\bx)-\nabla f_j(\bx')\|^2 \le L^2 \|\bx-\bx'\|^2,\forall j\\
&\|\nabla f(\bx)-\nabla f(\bx')\|^2 \le {L^2} \|\bx-\bx'\|^2.
\end{align}
\end{subequations}
\end{assumption}

\begin{assumption}  \label{A3}
	The mixing matrix $\bW\in\mathbb{R}^{m\times m}$ is symmetric, and satisfying the following
\begin{equation}\label{W}
	|\ubar{\lambda}_{\max}(\bW)|:=\eta<1,\quad  \bW \bOne=\bOne,
\end{equation}
where $\ubar{\lambda}_{\max}(\bW)$ denotes the second largest eigenvalue of $\bW$.
\end{assumption}

Note that many choices of mixing matrices satisfy the above condition. Here we give three commonly used mixing matrices \cite{xiao2004fast, boyd2004fastest}, where $d_i$ denotes the degree of node $i$, and $d_{\max} = \max_i \{d_i\}$:
\begin{itemize}
\item  Metropolis-Hasting Weight 
\begin{align}
    w_{ij}=\begin{cases}
    \frac{1}{1+\max\{d_i, d_j\}}, &\text{if}\quad \{i, j\}\in \mathcal{E},\\
    1 - \sum_{\{i, k\}\in \mathcal{E}} w_{ik}, &\text{if}\quad i=j,\\
    0, & \text{otherwise}.
    \end{cases}
\end{align}

\item  Maximum-Degree Weight
\begin{align}
    w_{ij}=\begin{cases}
    \frac{1}{d_{\max}}, &\text{if} \quad\{i, j\}\in \mathcal{E},\\
    1 - \frac{d_i}{d_{\max}}, &\text{if}\quad i=j,\\
    0, & \text{otherwise}.
    \end{cases}
\end{align}

\item Laplacian Weight  

 \begin{align}
    w_{ij}=\begin{cases}
    \gamma &\text{if}\quad \{i, j\}\in \mathcal{E},\\
    1- \gamma d_i, , &\text{if}\quad i=j,\\
    0 & \text{otherwise}.
    \end{cases}
\end{align}
If we use $\mathcal{L}$ to denote the graph Laplacian matrix, and $\lambda_{\max}, \ubar{\lambda}_{\min}$ as the largest and second smallest eigenvalue, then one of the common choices of $\gamma$ is $\frac{2}{\lambda_{\max}(\mathcal{L})+\ubar{\lambda}_{\min}(\mathcal{L})}$.  
 \end{itemize}

Next, let us formally define our communication and sample complexity measures.

\begin{Definition} {\bf(Sample Complexity)} The Incremental First-order Oracle (IFO) is defined as an operation in which, one node $i\in [m]$ takes a data sample $j \in [n]$, a point $\bw\in \mathbb{R}^d$, and returns the pair $(f^i_j(\bw), \nabla f^i_j(\bw))$.  The sample complexity  is defined as the total number of IFO calls required {across the entire network} to achieve an $\epsilon$ stationary solution defined in \eqref{eq:stationarity}.
\end{Definition}

\begin{Definition} {\bf(Communication Complexity)} 
In one round of communication, each node $i\in [m]$ is allowed to broadcast and received one $d$-dimensional vector to and from its neighbors, respectively.  Then the communication complexity is defined as the total rounds of communications  required to achieve an $\epsilon$ stationary solution defined in \eqref{eq:stationarity}. 
\end{Definition}

\subsection{Algorithm Design}\label{Sec:Algorithm}

In this section, we introduce the proposed  algorithm named Decentralized Gradient Estimation and Tracking (D-GET), for solving problem \eqref{P2_finite}.  
To motivate our algorithm design, we can observe from our discussion in Section \ref{sec:related} that,   the existing deterministic decentralized methods typically suffer from the high sample complexity, while the decentralized stochastic algorithms suffer from the high communication cost. Such a phenomenon inspires us to find a solution in between, which could simultaneously reduce the sample and the communication costs. 

{One natural solution is to incorporate the modern variance reduction techniques into the classical decentralized methods. Our idea is to use some variance reduced gradient estimator to track the full gradient of the entire problem, then perform decentralized gradient descent update. The gradient tracking step gives us fast convergence with a constant stepsize, while the variance reduction method significantly reduces the variation of the estimated gradient.}

Unfortunately, the decentralized methods and variance reduction techniques cannot be directly combined. Compared with the existing decentralized and variance reduction techniques in the literature, the key challenges in the algorithm design and analysis are given below: 
\begin{itemize}
	\item Due to the decentralized nature of the problem, none of the nodes can access the full gradient of the original objective function. The (possibly uncontrollable)  network consensus error always exists during the whole process of implementing the decentralized algorithm. Therefore, it is not clear that the existing variance reduction methods could be applied at each individual node effectively, since all of those require accurate global  gradient evaluation from time to time.
	\item It is then natural to integrate some procedure that is able to approximate the global gradient. For example, one straightforward way to perform gradient tracking is to introduce a new auxiliary variable $\by$ as the following \cite{di2016next, gnsd19}, which is updated by only using local estimated gradient and neighbors' parameters:
	\begin{align}\label{update:y:2}
	\by^{r}_i=\sum_{k\in\mathcal{N}_i}\bW_{ik}\by^{r-1}_k+  \frac{1}{|S^r_2|} \sum_{j\in S^r_2} \nabla f^i_j(\bx^{r}_i) -  \frac{1}{|S^{r-1}_2|} \sum_{j\in S^{r-1}_2} \nabla f^i_j(\bx^{r-1}_i)
	\end{align}
	where $S^r_2$ and $S^{r-1}_2$ are the samples selected at the $r$ and $r-1$th iterations, respectively.  
	If the tracked $\by_i$'s were used in the  (local) variance reduction procedure, there would be at least two main issues of decreasing the variance resulted from the tracked gradient as follows: 
	\emph{i}) at the early stage of implementing the decentralized algorithm,  the consensus/tracking error may dominate the variance of the tracked gradient, since the message of the full gradient has not been sufficiently propagated through the network. Consequently, performing variance reduction on $\by_i$'s will not be able to increase the quality of the full gradient estimation;
	\emph{ii})  even assuming that there was no consensus error. Since only the stochastic gradients, i.e., $\sum_{j\in\mathcal{S}^r_2}\nabla f^i_j(\bx^r_i)$, were used in the tracking, the $\by^r_i$'s themselves had high variance, 
	resulting that such (possibly low-quality) full gradient estimates may not be compatible to variance reduction methods as developed in the current literature (which often require full gradient evaluation from time to time). 
\end{itemize}

 The challenges discussed above suggest that it is non-trivial to design an algorithm that can be implemented in a fully decentralized manner, while still achieving the superior sample complexity and convergence rate achieved by state-of-the-art variance reduction methods. 
 In this work, we propose an algorithm which uses a novel decentralized gradient estimation and tracking strategy, together with a number of other design choices, to address the issues raised above. %In the following we present the key steps of performing D-GET in details.

To introduce the algorithm, let us first define two auxiliary local variables $\bv_i$ and $\by_i$, where $\bv_i$ is designed to estimate the local full batch gradient $\frac{1}{n} \sum_{j=1}^n \nabla f^i_j(\bx_i)$ by only using sample gradient $\nabla f^i_j(\bx_i)'s$, while $\by_i$ is designed  to track the global average gradient $\frac{1}{mn} \sum_{i=1}^m \sum_{j=1}^n \nabla f^i_j(\bx_i)$ by utilizing $\bv_i$'s. After the local and global gradient estimates are obtained, the algorithm performs local update based on the direction of $\by_i$; see the main steps below.  

\begin{itemize}
    \item  Local update using estimated gradient ($\bx$ update): Each local node $i$ first combines its previous iterates $\bx_i^{r-1}$ with its local neighbors $\bx_k^{r-1},\; k\in \mathcal{N}_i$ (by using the $k$th row of weight matrix $\mathbf{W}$), then makes a prediction based on the gradient estimate $\mathbf{y}^{r-1}_i$, i.e.,
\begin{align}\label{update:x}
\bx^{r}_i=\sum_{k\in\mathcal{N}_i}\bW_{ik}\bx^{r-1}_k-\alpha  \by^{r-1}_i.
\end{align}

\item Estimate local gradients ($\bv$ update): Each local node $i$ either directly calculates the full local gradient $\nabla f^i(\bx^r_i)$, or estimates its local gradient via an estimator $\bv$ using $|S_2|$ random samples, depending on the iteration $r$, i.e., 
\begin{align}\label{update:v}
\begin{cases}
\begin{aligned}
\bv_i^{r} &= \nabla f^i(\bx_i^{r}), && \text{mod}(r,q)=0\\
\bv_i^{r} &=  \frac{1}{|S_2|} \sum_{j\in S_2} \left[ \nabla f_j^i(\bx_i^{r}) - \nabla f_j^i(\bx_i^{r-1})  \right]+\bv_i^{r-1}, && \text{mod}(r,q)\ne 0
\end{aligned}
\end{cases}
\end{align}
where $q>0$ is the interval in which local full gradient will be evaluated once.

\item  Track global gradients ($\by$ update): Each local node $i$ combines its previous local estimate  $\by^{r-1}_i$ with its local neighbors $\by_k^{r-1}, k\in \mathcal{N}_i$, then makes a new estimation based on the fresh information $\bv_i^r$, i.e.,
\begin{align}\label{update:y}
\by^{r}_i=\sum_{k\in\mathcal{N}_i}\bW_{ik}\by^{r-1}_k+\bv_i^{r}-\bv_i^{r-1}.
\end{align}
\end{itemize}

In the following table, we summarize the proposed algorithm in a more compact form. %Our algorithm can be summarized in the 
Note that we use $\bx\in \mathbb{R}^{md}$, $\bv\in \mathbb{R}^{md}$, $\by\in \mathbb{R}^{md}$, $\nabla f(\bx) \in \mathbb{R}^{md}$ and $\nabla f_j(\bx)\in \mathbb{R}^{md}$ to denote the concatenation of the $\bx_i \in \mathbb{R}^{d}$, $\bv_i \in \mathbb{R}^{d}$, $\by_i \in \mathbb{R}^{d}$, $\nabla f^i(\bx_i)\in \mathbb{R}^{d}$ and $\nabla f^i_j(\bx_i) \in \mathbb{R}^{d}$ across all nodes.

\begin{algorithm}[h]
	\caption{{D-GET Algorithm for the finite sum problem \eqref{P2_finite}}}
	\label{Algorithm_Finite}
	\begin{algorithmic}
		\State {\bfseries Input:}  $\bx^0, \alpha, q, |S_2|$
		\State $\bv^{0} =  \nabla f(\bx^{0})$, $\by^0 = \nabla f(\bx^{0})$
		\For{$r=1,2,\ldots$}
		\State $\bx^{r}= \bW \bx^{r-1}-\alpha \by^{r-1}$ \Comment{local communication \& update} 
		\If {mod$(r,q)=0$}
		\State Calculate the full gradient
		\State $\bv^{r} =  \nabla f(\bx^{r})$
		\Comment{local gradient computation} 
		\Else
		\State Each node draws $S_2$ samples from $[n]$ with replacement
		\State $\bv^{r}=  \frac{1}{|S_2|} \sum_{j\in S_2} \left[ \nabla f_j(\bx^{r}) - \nabla f_j(\bx^{r-1})  \right]+\bv^{r-1}$
		\Comment{local gradient estimation} 
		\EndIf
		\State $\by^{r}= \bW \by^{r-1}+\bv^{r}-\bv^{r-1}$
		\Comment{global gradient tracking}
		\EndFor
	\end{algorithmic}
\end{algorithm}

\remark The above algorithm can also be interpreted as a ``double loop" algorithm, where each outer iteration (i.e., mod $(r, q)=0$) is followed by $q-1$  inner iterations (i.e., mod $(r, q)\ne 0$). The inner loop aims to estimate the local gradient via stochastic sampling at every iteration $r$, while the outer loop aims to reduce the estimation variance by recalculating the full batch gradient at every $q$ iterations. The local communication, update, and tracking steps are performed at both inner and outer iterations.

\remark {We further remark that, in D-GET, the total communication rounds is in the same order as the total number of iterations, since only two rounds of communications are performed per iteration, via broadcasting the local variable $\bx_i^{r-1}$ and $\by_i^{r-1}$ to their neighbors, and combining local $\bx^{r-1}_k$ and $\by^{r-1}_k$'s, $k\in \mathcal{N}_i$. On the other hand, the total number of samples used per iteration is either  $m|S_2|$ (where inner iterations are executed) or $mn$ (where outer iterations are executed)}.

\remark Note that our $\bx$ and $\by$ updates are reminiscent of the classical gradient tracking methods \cite{di2016next, gnsd19}, and $\bv$ update takes a similar form as the SARAH/SPIDER estimator \cite{nguyen2017sarah, fang2018spider}.
However, it is non-trivial to directly combine the gradient tracking and the variance reduction together, as we mentioned at the beginning of Section \ref{Sec:Algorithm}. The proposed D-GET uses a number of design choices to address these challenges.   For example, two vectors $\bv$ and $\by$ are used to respectively estimate the local and global gradients, in a way that the local gradient estimates do not depend on the (potentially inaccurate) global tracked gradients; to reduce the variance in $\by$, we occasionally use the full local gradient to perform tracking, etc. Nevertheless, the key challenge in the analysis is to properly bound the accumulated errors from the two estimates $\bv$ and $\by$.

\subsection{Convergence Analysis}\label{Convergence_Analysis_FS}

To facilitate our analysis, we first define the average iterates $\bar{\bx}$ and $\bar{\by}$ among all $m$ nodes,
\begin{subequations}
\begin{align}\label{def_bar}
&\bar{\bx}^r:=\frac{1}{m}\bOne^{\T}\bx^r=\frac{1}{m}\sum_i \bx_i^r, \\
&\bar{\bv}^r:=\frac{1}{m}\bOne^{\T}\bv^r=\frac{1}{m}\sum_i \bv_i^r,\\ &\bar{\by}^r:=\frac{1}{m}\bOne^{\T}\by^r=\frac{1}{m}\sum_i \by_i^r.
\end{align}
\end{subequations}
Note that here we use $r$ to denote the overall iteration number. By the double loop nature of the algorithm, {we can define the total number of outer iterations until iteration $r$ as below:}
%outer loop index as below
\begin{align*}
n_r := \lfloor r/q \rfloor+1, \;\;  \mbox{\rm such that}\;\;  (n_r - 1)q\le r\le n_rq-1.
\end{align*}

Before we formally conduct the analysis, we note three simple facts about Algorithm \ref{Algorithm_Finite}. 

First, {according to \eqref{update:x} and the definition \eqref{def_bar}}, the update rule of the average iterates can be expressed as: 
\begin{align}\label{x-update}
\bar{\bx}^{r}&=\bar{\bx}^{r-1}- \alpha \bar{\by}^{r-1}.
\end{align}

Second, if the iteration $r$ satisfies mod$(r,q)=0$ (that is, when the outer iteration is executed), from \eqref{update:v} and \eqref{update:y} it is easy to check that the 
following relations hold (given $\bv^0=\by^0$):  
\begin{align}\label{y_v_equal_finite}
&\bar{\bv}^r=\frac{1}{m} \sum_{i=1}^m \nabla f^i({\bx}_i^r)=\frac{1}{mn}\sum_{i=1}^m \sum_{k=1}^n  \nabla f^i_k(\bx_i^{r}),\\
&\bar{\by}^r= \bar{\bv}^r, \; \mbox{if}\; \mbox{\rm mod}(r,q)=0.
\end{align}
Third, if mod$(r,q)\ne0$, we have the following relations: 
\begin{align}
\bar{\bv}^{r} &= \frac{1}{m|S_2|} \sum_{i=1}^m   \sum_{j\in S_2} \left[ \nabla f^i_j(\bx_i^{r}) - \nabla f^i_j(\bx_i^{r-1})  \right]+\bar{\bv}^{r-1}, \label{y-update} \\
\bar{\by}^{r}&=\bar{\by}^{r-1}+\bar{\bv}^{r}-\bar{\bv}^{r-1}, \; \mbox{if}\; \mbox{\rm mod}(r,q)\ne 0.  \label{y-update_2}
\end{align}

Next, we outline the proof steps of the convergence rate analysis. 

\noindent\textbf{Step 1}. We first show that the variance of our local and global gradient estimators  can be bounded via $\bx$ and $\by$ iterates. The  bounds to be given below is  tighter than the classical analysis of decentralized stochastic methods, which assume the variance are bounded by some universal constant \cite{tang2018d,lian2017can,jiang2017collaborative}. This is an important step to obtain lower sample/communication complexity, since later we can show that the right-hand-side (RHS) of our bound shrinks as the iteration progresses.

\begin{lemma}(Bounded Variance) \label{bounded_variance}
	Under Assumption \ref{A1} - \ref{A3}, the sequence generated by the inner loop of Algorithm  \ref{Algorithm_Finite} satisfies the following inequalities (for all $(n_r - 1)q \le r \le n_rq - 1$):
	\begin{align}
	\mathbb{E}\|\bar{\by}^r -  \frac{1}{m} \sum_{i=1}^m \nabla f^i({\bx}_i^r)\|^2 &\le  \frac{8 L^2}{m|S_2|}   \sum_{t=(n_r-1)q}^{r}  \mathbb{E}\|  {\bx}^{t}  - \bOne \bar{\bx}^{t}\|^2 +\frac{4 \alpha^2 L^2 }{m|S_2|}   \sum_{t=(n_r-1)q}^{r} \mathbb{E} \|     \by^{t} - \bOne \bar{\by}^{t} \|^2 \label{lemma1-1}\\
	&+\frac{4 \alpha^2 L^2 }{|S_2| }    \sum_{t=(n_r-1)q}^{r} \mathbb{E}  \|     \bar{\by}^{t} \|^2+\mathbb{E} \|\bar{\by}^{(n_r-1)q} -\frac{1}{m} \sum_{i=1}^m      \nabla f^i({\bx_i}^{(n_r-1)q})\|^2. \nonumber\\
	\mathbb{E}\| {\bv}^r -    \nabla f({\bx}^r)\|^2	& \le   \frac{L^2}{|S_2|} \sum_{t=(n_r - 1)q}^{r} \mathbb{E}   \|  \bx^{t+1}  - \bx^{t} \|^2+ \mathbb{E} \| {\bv}^{(n_r - 1)q} -  \nabla f({\bx}^{(n_r - 1)q})\|^2.\label{lemma1-2}
	\end{align}	
\end{lemma}

\noindent\textbf{Step 2}. We then study the descent on $\mathbb{E} [f(\bar{\bx}^{r})]$, which is the expected value of the cost function evaluated on the average iterates.  

\begin{lemma}(Descent Lemma)\label{Lemma-Descent}
	Suppose Assumptions \ref{A1} - \ref{A3} hold, and for any $r\ge 0$ satisfying mod$(r, q) = 0$, the following holds for some $\epsilon_1\ge 0$:
	\begin{align}\label{epsilon_1}
	\mathbb{E} \left[\|\bar{\by}^{r} -\frac{1}{mn} \sum_{i=1}^m    \sum_{k=1}^n \nabla f^i_k(\bx_i^r)\|^2\right] \le \epsilon_1.
	\end{align}
	Then by applying Algorithm  \ref{Algorithm_Finite}, we have the following relation  for all $r\ge 0$,
	\begin{align}\label{Lemma2}
	\mathbb{E} [f(\bar{\bx}^{r+1})] &\le \mathbb{E}[f(\bar{\bx}^0)] - \left(\frac{\alpha}{2}   - \frac{\alpha^2 L}{2} - \frac{4 \alpha^3 L^2 q}{|S_2|} \right) \sum_{t=0}^{r}\mathbb{E}\|\bar{\by}^t\|^2 \\
	&+ \frac{\alpha L^2}{m}  \sum_{t=0}^{r}\mathbb{E} \| {\bx^t} - \bOne \bar{\bx}^t \|^2 +  \frac{8 \alpha L^2q}{m|S_2|}  \sum_{t=0}^{r}  \mathbb{E} \|  {\bx}^{t}  - \bOne \bar{\bx}^{t}\|^2  +\frac{4 \alpha^3 L^2q }{m|S_2|}   \sum_{t=0}^{r}   \mathbb{E} \|     \by^{t} - \bOne \bar{\by}^{t}\|^2 +\alpha (r+1) \epsilon_1. \nonumber
	\end{align}
\end{lemma}

A key observation from  \leref{Lemma-Descent} is that, in the RHS of \eqref{Lemma2}, besides the negative term in $\mathbb{E}  \|     \bar{\by}^{k} \|^2$, we also have several extra terms (such as $\mathbb{E}\|  {\bx}^{k}  - \bOne \bar{\bx}^{k}\|^2$ and $\mathbb{E} \|     \by^{k} - \bOne \bar{\by}^{k} \|^2$) that cannot be made negative. Therefore, we need to find some potential function that is strictly descending per iteration.

{Note that $\epsilon_1$ in \eqref{epsilon_1} comes from the variance of $\bv^r$ in estimating the full local gradient at each outer loop $n_r$. For Algorithm \ref{Algorithm_Finite}, where we calculate a full batch gradient per outer loop in step \eqref{y_v_equal_finite}, it is clear that $\epsilon_1 = 0$. However, we still would like to include $\epsilon_1$ in the above result because, later when we analyze the online version (where such a variance will no longer be zero), we can re-use the above result.}

\noindent\textbf{Step 3}. Next, we introduce the contraction property, which combined with  $\mathbb{E} [f(\bar{\bx}^{r})]$ will be used to construct the potential function. 
\begin{lemma}(Iterates Contraction)\label{Lemma-Contraction}
	Using the Assumption \ref{A3} on $\bW$ and applying Algorithm  \ref{Algorithm_Finite}, we have the following contraction property of the iterates:
	\begin{align}
	&\mathbb{E} \|\bx^{r+1}-\bOne\bar{\bx}^{r+1}\|^2 \le (1+\beta)\eta^2 \mathbb{E}\| \bx^r -\bOne\bar{\bx}^r\|^2+(1+\frac{1}{\beta})\alpha^2 \mathbb{E}\|\by^r-\bOne\bar{\by}^r\|^2, \label{x_contraction}\\
	& \mathbb{E}\|\by^{r+1}-\bOne \bar{\by}^{r+1}\|^2
	\le (1+\beta) \eta^2\mathbb{E} \|\by^{r} - \bOne  \bar{\by}^r\|^2 +(1+\frac{1}{\beta}) \mathbb{E} \|{\bv}^{r+1} -  {\bv}^{r} \|^2,\label{y_contraction}
	\end{align}
	where $\beta$ is some constant such that $(1+\beta)\eta^2<1$.
	
	If we further assume for all $r\ge 0$ satisfying mod$(r, q)=0$, the following holds for some $\epsilon_2\ge 0$:
	\begin{align}	\label{epsilon_2}
	\mathbb{E}\|\bv^{r}-\nabla f(\bx^{r})\|^2 \le \epsilon_2.
	\end{align}  
	Then we  have the following bound on successive difference of $\bv$ for all $r\ge 0$:	
	\begin{multline}
	\sum_{t=0}^r \mathbb{E}\|\bv^{t+1}-\bv^{t}\|^2  \le  24L^2 \left( 2\sum_{t=0}^r \mathbb{E} \|   \bx^{t} - \bOne \bar{\bx}^{t} \|^2 +  \alpha^2  \sum_{t=0}^r \mathbb{E}\| \by^t - \bOne \bar{\by}^t\|^2+  \alpha^2  \sum_{t=0}^r \mathbb{E}\|  \bOne \bar{\by}^t\|^2\right)  
	+ 6(r+1)\epsilon_2.\label{v_contraction}
	\end{multline}
\end{lemma}

{Again, $\epsilon_2$ comes from the variance of the estimating the local gradient in each outer loop, and we have $\epsilon_2 = 0$ for Algorithm \ref{Algorithm_Finite}.}
{Note that \eqref{x_contraction} can also be written as following
\begin{align}
	\mathbb{E} \|\bx^{r+1}-\bOne\bar{\bx}^{r+1}\|^2 -  \mathbb{E}\| \bx^r -\bOne\bar{\bx}^r\|^2 \le \left((1+\beta)\eta^2-1\right) \mathbb{E}\| \bx^r -\bOne\bar{\bx}^r\|^2+(1+\frac{1}{\beta})\alpha^2 \mathbb{E}\|\by^r-\bOne\bar{\by}^r\|^2. \label{x_contraction_shrink}
\end{align}
One key observation here is that we have $(1+\beta)\eta^2-1<0$ by properly choosing $\beta$. Therefore, the RHS of the above equation can be made negative by properly selecting the stepsize $\alpha$. }

\noindent\textbf{Step 4}. 
This step combines the descent estimates obtained in Step 2-3 to construct a potential function, by using a conic combination of $\mathbb{E} [f(\bar{\bx}^{r})]$,  $\mathbb{E} \|\bx^{r}-\bOne\bar{\bx}^{r}\|^2$ and $\mathbb{E} \|\by^{r}-\bOne\bar{\by}^{r}\|^2$.

 \begin{lemma}(Potential Function)\label{Lemma-Potential}
	Constructing the following potential function 
	\begin{align*}
	\textsl{H}(\bx^r) & := \mathbb{E} [f(\bar{\bx}^{r})]+ {\frac{1}{m}}\mathbb{E} \|\bx^r-\bOne\bar{\bx}^r\|^2+ {\frac{\alpha }{m}}   \mathbb{E}  \|{\by}^r-\bOne{\bar{\by}}^r\|^2.
	\end{align*}
	Under Assumption \ref{A1} - \ref{A3} and  Algorithm  \ref{Algorithm_Finite}, if we further pick $q=|S_2|$ and {define $\epsilon_1$ and $\epsilon_2$  as  in \eqref{epsilon_1} and \eqref{epsilon_2}}, we have
	\begin{align*}
	&\textsl{H}(\bx^{r+1}) -  \textsl{H}(\bx^0) \le- C_1 \sum_{t=0}^{r}\mathbb{E}\|\bar{\by}^t\|^2  - C_2  \sum_{t=0}^{r}{\frac{1}{m}}\mathbb{E}\| \bx^t -\bOne\bar{\bx}^t\|^2 -C_3 \sum_{t=0}^{r} {\frac{1}{m}}\mathbb{E}\|\by^t-\bOne\bar{\by}^t\|^2 +\epsilon_3,
	\end{align*}
	where
	\begin{align}
	&C_1:= \left(\frac{1}{2}   - \frac{\alpha L}{2} - {4 \alpha^2 L^2 } -24(1+\frac{1}{\beta}) \alpha^2L^2\right)\alpha, \label{C1}\\
	&C_2:=\left( 1 - (1+\beta)\eta^2 - 48\alpha(1+\frac{1}{\beta}) L^2 - {9 \alpha L^2}\right),\label{C2}\\
	&C_3:=\left( 1- (1+\beta) \eta^2-(1+\frac{1}{\beta})\alpha -24(1+\frac{1}{\beta}) \alpha^2L^2 - {4 \alpha^2 L^2 }  \right)\alpha,\label{C3}\\
	&\epsilon_3:= \alpha (r+1) (\epsilon_1+6{\frac{1}{m}}(1+\frac{1}{\beta}) \epsilon_2).
	\end{align}
\end{lemma}
\noindent\textbf{Step 5}. We can then properly choose the stepsize $\alpha$, and make $C_1, C_2, C_3$ to be positive. Therefore,  our solution quality measure $\mathbb{E}\|\frac{1}{m} \sum_{i=1}^m \nabla f^i({\bx}_i^r)\|^2 +   {\frac{1}{m}}  \mathbb{E}\|\bx^r-\bOne\bar{\bx}^r\|^2$ can be expressed as the difference of the potential functions and the proof is complete.\\

\begin{theorem}\label{th:main}
	Consider problem \eqref{P2_finite} and under Assumption \ref{A1} - \ref{A3}, if we pick $\alpha = \min \{K_1, K_2, K_3\}$ and $q=|S_2|=\sqrt{n}$, then we have following results by applying Algorithm \ref{Algorithm_Finite},
	\begin{align*}
	\min_{r\in [T]} \mathbb{E}\|\frac{1}{m} \sum_{i=1}^m \nabla f^i({\bx}_i^r)\|^2 +  {\frac{1}{m}}   \mathbb{E}\|\bx^r-\bOne\bar{\bx}^r\|^2 \le C_0  \cdot  \frac{\mathbb{E}f({\bx}^{0}) - \ubar{f}}{T},
	\end{align*}
	where $\ubar{f}$ denotes the lower bound of $f$, and the constants are defined as following
	\begin{align*}
	K_1:= & \frac{-\frac{L}{2}+\sqrt{(\frac{L}{2})^2+48(1+\frac{1}{\beta})L^2+{8L^2}}}{48(1+\frac{1}{\beta})L^2+8L^2}, \\
	K_2:=&\frac{1-(1+\beta)\eta^2}{48(1+\frac{1}{\beta})L^2+{9L^2}}, \\
	K_3:=&\frac{-(1+\frac{1}{\beta})+\sqrt{(1+\frac{1}{\beta})^2+4(1-(1+\beta)\eta^2)(24(1+\frac{1}{\beta})+{4L^2})}}{48(1+\frac{1}{\beta})L^2+{8L^2}},\\
	C_0:=& \left(   \frac{8 \alpha^2 L^2+2}{C_1} + \frac{16 L^2+1}{m C_2}  +\frac{8 \alpha^2 L^2  }{m C_3}  \right),
	\end{align*}
	in which $\eta$ denotes the second largest eigenvalue of the mixing matrix from \eqref{W}, $\beta$ denotes a constant satisfying $1-(1+\beta)\eta^2>0$, and $C_1, C_2, C_3$ are defined in \eqref{C1}-\eqref{C3}.
\end{theorem}

By directly applying the above result, we have the upper bound on gradient and communication cost by properly choosing $T$ based on $\epsilon$.
\begin{Corollary} \label{Corollary1}
	To achieve the following $\epsilon$ stationary solution of problem \eqref{P2_finite} by Algorithm \ref{Algorithm_Finite}, 
	\begin{align*}
	\min_{r\in [T]} \mathbb{E}\|\frac{1}{m} \sum_{i=1}^m \nabla f^i({\bx}_i^r)\|^2 +  {\frac{1}{m}} \mathbb{E} \|\bx^r-\bOne\bar{\bx}^r\|^2 \le \epsilon,
	\end{align*}
	{the total number of iterations $T$ and communication rounds required are both in the order of
	$\mathcal{O}({\epsilon^{-1}})$,
	and the total number of samples evaluated across the network is in the order of	
	$\mathcal{O}(mn+{m{n}^{1/2}}{\epsilon^{-1}})$.}
\end{Corollary}
 \begin{proof}
	If we pick $T=\lfloor C_0 \cdot \frac{ \mathbb{E}f({\bx}^{0}) - \ubar{f} }{\epsilon}\rfloor + 1 \ge C_0\cdot \frac{  \mathbb{E}f({\bx}^{0}) - \ubar{f}  }{\epsilon}$, then  we can obtain following from \thref{th:main}
	\begin{align*}
	\min_{r\in [T]} \mathbb{E}\|\frac{1}{m} \sum_{i=1}^m \nabla f^i({\bx}_i^r)\|^2 +   {\frac{1}{m}}\mathbb{E}\|\bx^r-\bOne\bar{\bx}^r\|^2 \le C_0  \cdot  \frac{\mathbb{E}f({\bx}^{0}) - \ubar{f}}{T} \le \epsilon.
	\end{align*}
	
	Therefore, the total  {samples needed} will be the {sum} of outer loop complexity ($\lceil  \frac{T}{q}   \rceil$ times full ($n$) gradient evaluations per node) plus inner loop complexity ({$T$ times $|S_2|$ gradient evaluations per node}), by letting $q=|S_2|=\sqrt{n}$, we conclude that the total samples needed are 
	\begin{align*}
	& m\times \left(\lceil  \frac{T}{q}   \rceil \cdot n + T \cdot |S_2|\right) \nonumber\\
	& \le m\times \left(\left(\frac{T}{\sqrt{n}}+1\right)n + T\sqrt{n} \right)\\
	&\le m\left( n + 2C_0 \cdot \frac{\mathbb{E}f({\bx}^{0}) - \ubar{f} }{\epsilon} \sqrt{n} + 2\sqrt{n}\right)\nonumber\\
	& = \mathcal{O} \left(m\times \left(n+\frac{\sqrt{n}}{\epsilon}\right)\right).
	\end{align*}
	This completes the proof.
\end{proof}

\section{The Online Setting}\label{Section_Online}
In this section, we discuss the online setting \eqref{eq:online} for solving problem \eqref{P2}, where the problem can either be expressed as the following
\begin{align}\label{P2_online}
\min_{\bx \in \mathbb{R}^{md}} f(\bx)=\frac{1}{m}  \sum_{i=1}^m  \mathbb{E}_{\xi\sim\mathcal{D}_i} \left[f^i_{\xi}(\bx_i)\right],
\quad\textrm{s.t.}\quad  \bx_i=\bx_j, \quad\forall (i, j)\in \mathcal{E}, \tag{P2}
\end{align}
where $\xi$ represents the data drawn from the data distribution $\mathcal{D}_i$ at the $i$th node, or in form \eqref{P2_finite} such that the number of samples $n$ is too large to calculate the full batch even occasionally. In either one of these scenarios, full batch evaluations at the local nodes are no longer performed for each outer iteration. 

The above setting has been well-studied for the centralized problem (with large or even infinite number of samples). For example, in SCSG \cite{lei2017non}, a batch size  $\mathcal{O}(\epsilon^{-1})$ is used when the sample size is large or the target accuracy $\mathcal{O}(\epsilon)$ is moderate, improving the rate to $\mathcal{O}(\epsilon^{-5/3})$ from $\mathcal{O}(\epsilon^{-2})$ compared to the vanilla SGD \cite{ghadimi2013stochastic}. Recently, SPIDER \cite{fang2018spider} further improves the results to $ \mathcal{O}(\epsilon^{-3/2})$, while the SpiderBoost \cite{wang2018spiderboost} uses  a constant stepsize and is amenable to solve non-smooth problem at this rate.  
 
\subsection{The Proposed Algorithm}
To begin with, we first introduce two commonly used assumptions in the online learning setting.

\begin{assumption}  \label{A2}
	At each iteration, samples are independently collected, and the stochastic gradient is an unbiased estimate of the true gradient:
	\begin{align}\label{assume-unbiased}
	&\mathbb{E}_{\xi}[  \nabla  f_{\xi}^i(\bx_i)]=\nabla f^i(\bx_i),\forall i.
	\end{align}
\end{assumption}

\begin{assumption}   \label{A4}
	The variance between the stochastic gradient and the true gradient is bounded:
	\begin{align}\label{assume-bounded}
	&\mathbb{E}_{\xi}[  \|\nabla  f_{\xi}^i(\bx_i) - \nabla f^i(\bx_i) \|^2]\le \sigma^2,\forall i.
	\end{align}
	\end{assumption}
 
To present our algorithms, note that compared to problem \eqref{P2_finite}, the main difference of having the expectation in \eqref{P2_online} is that the full batch gradient evaluation is  no longer feasible. Therefore, we need to slightly revise our algorithm in Section \ref{Section_FiniteSum} and redesign the local gradient estimation step (i.e., the $\bv$ update).  Specifically, different from \eqref{update:v} where we sample the full batch, here we randomly draw $S_1$ samples, the size of which is inversely proportional to the desired accuracy $\epsilon$. We have the following updates on $\bv$:
\begin{align} \label{update:v:online}
\begin{cases}
\begin{aligned}
\bv_i^{r} &= \frac{1}{|S_1|} \sum_{\xi\in S_1}  \nabla f^i_\xi(\bx_i^{r}), && \text{mod}(r,q)=0\\
\bv_i^{r} &=  \frac{1}{|S_2|} \sum_{\xi\in S_2} \left[ \nabla f_\xi^i(\bx_i^{r}) - \nabla f_\xi^i(\bx_i^{r-1})  \right]+\bv_i^{r-1}, && \text{mod}(r,q)\ne 0.
\end{aligned}
\end{cases}
\end{align}
It is easy to check that the following relation on average iterates is obvious when mod$(r, q)=0$ and $\bar{\by}^0=\bar{\bv}^0$, 
\begin{align}\label{y_v_equal_online}
 	\bar{\by}^r=\bar{\bv}^r =\frac{1}{m|S_1|}\sum_{i=1}^m \sum_{\xi\in S_1}  \nabla f^i_\xi(\bx_i^{r}).
\end{align}
The rest of the  updates on $\bx$ and $\by$ are same as the finite sum setting;  see  Algorithm \ref{Algorithm_Online} below for details. 
\begin{algorithm}[h]
	\caption{{D-GET Algorithm (global view) (online)}}
	\label{Algorithm_Online}
	\begin{algorithmic}
		\State {\bfseries Input:}  $\bx^0, \alpha, q, |S_1|, |S_2|$
		\State Draw $S_1$ samples with replacement
		\State $\bv^{0} = \frac{1}{|S_1|} \sum_{\xi\in S_1}  \nabla f_\xi(\bx^{0})$, $\by^0 = \bv^0$
		\For{$r=1,2,\ldots$}
		\State $\bx^{r}= \bW \bx^{r-1}-\alpha \by^{r-1}$ \Comment{local communication \& update}
		\If {mod$(r,q)=0$}
		\State Draw $S_1$ samples with replacement
		\State $\bv^{r} = \frac{1}{|S_1|} \sum_{\xi\in S_1}  \nabla f_\xi(\bx^{r})$  \Comment{local gradient estimation}
		\Else
		\State Draw $S_2$ samples with replacement
		\State $\bv^{r}=  \frac{1}{|S_2|} \sum_{\xi\in S_2} \left[ \nabla f_\xi(\bx^{r}) - \nabla f_\xi(\bx^{r-1})  \right]+\bv^{r-1} $   \Comment{local gradient estimation}
		\EndIf
		\State $\by^{r}= \bW \by^{r-1}+\bv^{r}-\bv^{r-1}$    \Comment{global gradient tracking}
		\EndFor
	\end{algorithmic}
\end{algorithm}

\subsection{Convergence Analysis}
The analysis follows the same steps as described in Section \ref{Convergence_Analysis_FS} and it is easy to verify that our \leref{bounded_variance} to \leref{Lemma-Potential}  still hold true for Algorithm \ref{Algorithm_Online}. 
However, for online setting where we no longer sample a full batch, the variance $\epsilon_1$ and $\epsilon_2$ cannot be eliminated.  The lemma given below provides the bounds on $\epsilon_1$ and $\epsilon_2$.

\begin{lemma}(Bounded Variance)\label{online_lemma5}
Under Assumption \ref{A1} to \ref{A4}, the sequence generated by the outer loop of Algorithm \ref{Algorithm_Online} satisfies the following relations (for all $r$ such that mod$(r, q)$=0)
	\begin{align*}
&\mathbb{E} \|{\bv}^{r} - \nabla f({\bx}^{r})\|^2 
\le  \frac{m\sigma^2}{|S_1|},\\
&	\mathbb{E} \|\bar{\by}^{r} -\frac{1}{m} \sum_{i=1}^m      \nabla f^i(\bx_i^r)\|^2 
	\le  \frac{\sigma^2}{|S_1|}.
	\end{align*}
\end{lemma}	

By using the above lemma, we can then choose  the sample  size inversely proportional to the targeted accuracy and obtain our final results.

\begin{theorem}\label{th:main2}
	Suppose Assumption \ref{A1} - \ref{A4} hold, and pick the following parameters for problem \eqref{P2_online}:
	$$\alpha = \min \{K_1, K_2, K_3\}, \quad |S_1| = (4C_0\alpha(7+\frac{6}{\beta})\sigma^2+8\sigma^2)/{\epsilon}, \quad q=|S_2|=\sqrt{|S_1|}.$$ 
	Then we have the following result by applying Algorithm \ref{Algorithm_Online},
	\begin{align*}
	\min_{r\in [T]} \mathbb{E}\left\|\frac{1}{m} \sum_{i=1}^m \nabla f^i({\bx}_i^r)\right\|^2 +  {\frac{1}{m}}\mathbb{E}  \|\bx^r-\bOne\bar{\bx}^r\|^2 \le C_0  \cdot  \frac{\mathbb{E}f({\bx}^{0}) - \ubar{f}}{T} + \frac{\epsilon}{2}.
	\end{align*}
\end{theorem}

\begin{Corollary}\label{Corollary2}
	By using Algorithm \ref{Algorithm_Online}, to achieve the $\epsilon$ stationary solution of problem \eqref{P2}, i.e.,
	\begin{align*}
	\min_{r\in [T]} \mathbb{E}\|\frac{1}{m} \sum_{i=1}^m \nabla f^i({\bx}_i^r)\|^2 +    {\frac{1}{m}}\mathbb{E}\|\bx^r-\bOne\bar{\bx}^r\|^2 \le \epsilon,
	\end{align*}
	the total number of iterations $T$ and communication rounds required are both in the order of
	$\mathcal{O}({\epsilon^{-1}})$,
	and the total sample complexity is in the order of {$\mathcal{O}( m\epsilon^{-3/2})$}.
\end{Corollary}
\begin{proof}
	
	If we pick the following constants for Algorithm \ref{Algorithm_Online}: 
	$$|S_1| = \frac{4C_0\alpha(7+\frac{6}{\beta})\sigma^2+8\sigma^2}{\epsilon},\; q=|S_2|=\sqrt{|S_1|}, \; T=2\lfloor C_0 \cdot \frac{ \mathbb{E}f({\bx}^{0}) - \ubar{f} }{\epsilon}\rfloor + 2 \ge 2C_0\cdot \frac{  \mathbb{E}f({\bx}^{0}) - \ubar{f}  }{\epsilon}$$ then from \thref{th:main2} we can obtain
	\begin{align*}
	\min_{r\in [T]}\mathbb{E}\|\frac{1}{m} \sum_{i=1}^m \nabla f^i({\bx}_i^r)\|^2 +   {\frac{1}{m}}\mathbb{E}\|\bx^r-\bOne\bar{\bx}^r\|^2 \le \underbrace{C_0  \cdot  \frac{\mathbb{E}f({\bx}^{0}) - \ubar{f}}{T}}_{\le \frac{\epsilon}{2}} + \frac{\epsilon}{2} \le \epsilon.
	\end{align*}
	
	Therefore we have that the per-node sample evaluations are given as  
	\begin{align*}
	\lceil  \frac{T}{q}   \rceil \cdot |S_1| + T \cdot |S_2| \le \left( \frac{T}{\sqrt{|S_1|}} +1\right)|S_1| + T\sqrt{|S_1|} = O(\frac{1}{\epsilon} + \frac{1}{\epsilon^{3/2}}).
	\end{align*}

This completes the proof.
\end{proof}

\section{Concluding Remarks}
In this work, we proposed a joint gradient estimation and tracking approach (D-GET) for  fully decentralized non-convex optimization problems. By utilizing modern variance reduction and gradient tracking techniques, the proposed method improves the sample and/or communication complexities compared with existing methods. In particular, for decentralized finite sum problems, the proposed approach  requires only {$\mathcal{O}(mn^{1/2}\epsilon^{-1})$} sample complexity and $\mathcal{O}(\epsilon^{-1})$ communication complexity to reach the $\epsilon$ stationary solution. 
For online problem, our approach achieves an {$\mathcal{O}(m\epsilon^{-3/2})$} sample  and an $\mathcal{O}(\epsilon^{-1})$ communication complexity, which significantly improves upon the best existing bounds of  {$\mathcal{O}(m\epsilon^{-2})$ and $\mathcal{O}(\epsilon^{-2})$} as derived in \cite{tang2018d, gnsd19}.

%\newpage
{\small
\bibliographystyle{IEEEtran}
\bibliography{ref_VR,ref_Decentralized}}

 \clearpage
 \newpage
 \appendix
 
 In the appendix, we provide a complete theoretical analysis on the convergence of the proposed D-GET algorithm.

 \section{Proof of \leref{bounded_variance}}
 \begin{proof}
 	Define $\mathbb{E} [\cdot | \mathcal{F}_r]$ as the expectation with respect to the random choice of sample $j$, conditioned on $\bx^0, \cdots, \bx^{r}$, $\bv^0, \cdots, \bv^{r-1}$ and $\by^0, \cdots, \by^{r-1}$. 
 	
 	First, we have the following identity (which holds true when $\mbox{\rm mod}(r,q)\ne 0$)
 	\begin{align}\label{C_Expectation}
 	\mathbb{E} [\bar{\bv}^{r}-\bar{\bv}^{r-1} | \mathcal{F}_r] \stackrel{\eqref{y-update}}{=}  \mathbb{E} \left[ \frac{1}{m|S_2|} \sum_{i=1}^m \sum_{j\in S_2} \left[ \nabla f^i_j(\bx_i^{r}) - \nabla f^i_j(\bx_i^{r-1})  \right]   \bigg| \mathcal{F}_r\right] =   \frac{1}{m} \sum_{i=1}^m  \left[ \nabla f^i(\bx_i^{r}) - \nabla f^i(\bx_i^{r-1})  \right].
 	\end{align}
 	To see why the second equality holds, note that when $\bx^{r}, \by^{r-1}, \bv^{r-1}$ are known and fixed, the second expectation is taken over the random selection of $S_2$. The second equality follows because $S_2$ is sampled from $[n]$ uniformly with replacement, and it is an unbiased estimate of the averaged gradient.
 	
 	Second, it is straightforward to obtain the following equality,
 	\begin{align}
 	\norm{ \bar{\by}^r -  \frac{1}{m} \sum_{i=1}^m \nabla f^i({\bx}_i^r)}^2	  &\stackrel{\eqref{y-update_2}}{=}
 	\left \| \bar{\by}^{r-1}+\bar{\bv}^{r}-\bar{\bv}^{r-1}-\frac{1}{m} \sum_{i=1}^m \nabla f^i({\bx}_i^{r})\right \|^2 \nonumber\\
 	& =\left\|\bar{\by}^{r-1} -\frac{1}{m} \sum_{i=1}^m \nabla f^i({\bx}_i^{r-1})\right\|^2 + \left\| \bar{\bv}^{r}-\bar{\bv}^{r-1}+\frac{1}{m} \sum_{i=1}^m \nabla f^i({\bx}_i^{r-1}) -\frac{1}{m} \sum_{i=1}^m \nabla f^i({\bx}_i^{r})\right\|^2 \nonumber\\
 	&+2\left\langle \bar{\by}^{r-1} -\frac{1}{m} \sum_{i=1}^m \nabla f^i({\bx}_i^{r-1}),  \bar{\bv}^{r}-\bar{\bv}^{r-1}+\frac{1}{m} \sum_{i=1}^m \nabla f^i({\bx}_i^{r-1}) -\frac{1}{m} \sum_{i=1}^m \nabla f^i({\bx}_i^{r}) \right\rangle, \label{cross_term}
 	\end{align}
 	where in the second equality, we add and subtract a term $\frac{1}{m} \sum_{i=1}^m \nabla f^i({\bx}_i^{r-1})$.
 	
 	The  cross term in \eqref{cross_term} can be eliminated if we take the conditional expectation conditioning on  $\mathcal{F}_r$. Since under $\mathcal{F}_r$, we have $\bx^r, \bx^{r-1}, \bv^{r-1}, \by^{r-1}, \bar{\by}^{r-1}$ are all known and fixed. Further applying \eqref{C_Expectation} we have
 	\begin{align}\label{eq:unbiased}
 	\mathbb{E} \left[\bar{\bv}^{r}-\bar{\bv}^{r-1}+\frac{1}{m} \sum_{i=1}^m \nabla f^i({\bx}_i^{r-1}) -\frac{1}{m} \sum_{i=1}^m \nabla f^i({\bx}_i^{r})\bigg|\mathcal{F}_r\right] = 0.
 	\end{align}

 	Further taking the full expectation on \eqref{cross_term} we have %\mhcomment{change, by using definition \eqref{fintiesum}}
 	\begin{align}\label{bounded_ybar_per_iterates}
 	&\mathbb{E}\left\|\bar{\by}^r -  \frac{1}{m} \sum_{i=1}^m \nabla f^i({\bx}_i^r)\right\|^2	\stackrel{\eqref{eq:finite:sum}\eqref{y-update}\eqref{eq:unbiased}}{=}    \mathbb{E} \left\|\bar{\by}^{r-1} -\frac{1}{m} \sum_{i=1}^m   \nabla f^i({\bx_i}^{r-1})\right\|^2 \\
 	&+\mathbb{E} \left\|\frac{1}{m|S_2|} \sum_{i=1}^m  \sum_{j\in S_2}   \nabla f_j^i(\bx_i^{r})   -  \frac{1}{m|S_2|} \sum_{i=1}^m  \sum_{j\in S_2}  \nabla f_j^i(\bx_i^{r-1})+\frac{1}{mn} \sum_{i=1}^m    \sum_{k=1}^n \nabla f^i_k({\bx_i}^{r-1})- \frac{1}{mn} \sum_{i=1}^m    \sum_{k=1}^n \nabla f^i_k({\bx}_i^r) \right\|^2.\nonumber
 	\end{align}
 	
 	Since each $j\in S_2$ is sampled from $[n]$ uniformly, we have the following conditional expectation:
 	\begin{align} \label{cross_exp_0_new}
 	 \mathbb{E}_j\left[ \frac{1}{m} \sum_{i=1}^m    \nabla f_j^i(\bx_i^{r}) -  \frac{1}{m} \sum_{i=1}^m    \nabla f_j^i(\bx_i^{r-1})  - \frac{1}{mn} \sum_{i=1}^m    \sum_{k=1}^n \nabla f^i_k({\bx}_i^r) +\frac{1}{mn} \sum_{i=1}^m    \sum_{k=1}^n \nabla f^i_k({\bx_i}^{r-1}) \bigg| \mathcal{F}_r\right]=0.
 	 \end{align}
 	 
 	 Then the second term of RHS of \eqref{bounded_ybar_per_iterates}
can be further bounded through following (where $\mathbb{E}[\cdot]$ is the full expectation)
 	\begin{align}\label{bounded_ybar_via_x}
 	& \mathbb{E} \left\|\frac{1}{m|S_2|} \sum_{i=1}^m  \sum_{j\in S_2}   \nabla f_j^i(\bx_i^{r}) -  \frac{1}{m|S_2|} \sum_{i=1}^m  \sum_{j\in S_2}  \nabla f_j^i(\bx_i^{r-1})  - \frac{1}{mn} \sum_{i=1}^m    \sum_{k=1}^n \nabla f^i_k({\bx}_i^r) +\frac{1}{mn} \sum_{i=1}^m    \sum_{k=1}^n \nabla f^i_k({\bx_i}^{r-1})\right\|^2 \nonumber\\
 	  = & \frac{1}{|S_2|^2} \mathbb{E}    \left\| \sum_{j\in S_2} \left(\frac{1}{m} \sum_{i=1}^m    \nabla f_j^i(\bx_i^{r}) -  \frac{1}{m} \sum_{i=1}^m    \nabla f_j^i(\bx_i^{r-1})  - \frac{1}{mn} \sum_{i=1}^m    \sum_{k=1}^n \nabla f^i_k({\bx}_i^r) +\frac{1}{mn} \sum_{i=1}^m    \sum_{k=1}^n \nabla f^i_k({\bx_i}^{r-1})\right)\right\|^2 \nonumber \\
 	\stackrel{(i)} = & \frac{1}{|S_2|^2}   \sum_{j\in S_2}  \mathbb{E} \left\|\frac{1}{m} \sum_{i=1}^m    \nabla f_j^i(\bx_i^{r}) -  \frac{1}{m} \sum_{i=1}^m    \nabla f_j^i(\bx_i^{r-1})  - \frac{1}{mn} \sum_{i=1}^m    \sum_{k=1}^n \nabla f^i_k({\bx}_i^r) +\frac{1}{mn} \sum_{i=1}^m    \sum_{k=1}^n \nabla f^i_k({\bx_i}^{r-1})\right\|^2  \nonumber\\
 	\stackrel{(ii)} \le & \frac{1}{|S_2|}\mathbb{E} \left\|\frac{1}{m} \sum_{i=1}^m    \nabla f_j^i(\bx_i^{r}) -  \frac{1}{ m} \sum_{i=1}^m     \nabla f_j^i(\bx_i^{r-1}) \right\|^2 \nonumber\\
 	\stackrel{(iii)} \le & \frac{1}{m|S_2|}    \mathbb{E} \left[ \sum_{i=1}^m  \| \nabla f_j^i(\bx_i^{r}) -       \nabla f_j^i(\bx_i^{r-1}) \|^2 \right]\nonumber\\
 	\stackrel{(iv)} \le &  \frac{L^2}{m|S_2|}      \mathbb{E} \|  \bx^{r} -\bx^{r-1} \|^2, 
 	\end{align}	
 	In step $(i)$ of the above relation, we use the fact that for two random variables $u_i$, $u_j$ which are independent conditioning on $\mathcal{F}$, the following holds
 	\begin{align}
 	\mathbb{E}[\langle u_\ell, u_j\rangle] = \mathbb{E}_{\mathcal{F}}\mathbb{E}[\langle u_\ell, u_j\rangle\mid \mathcal{F}] = \mathbb{E}_{\mathcal{F}}\langle \mathbb{E}[u_\ell\mid \mathcal{F}], \mathbb{E}[u_j\mid \mathcal{F}]\rangle .
 	\end{align}
 	Plugging $u_j$ as below and $u_{\ell}$ similarly,
 	\begin{align}
 	u_j = \frac{1}{m} \sum_{i=1}^m    \nabla f_j^i(\bx_i^{r}) -  \frac{1}{m} \sum_{i=1}^m    \nabla f_j^i(\bx_i^{r-1})  - \frac{1}{mn} \sum_{i=1}^m    \sum_{k=1}^n \nabla f^i_k({\bx}_i^r) +\frac{1}{mn} \sum_{i=1}^m    \sum_{k=1}^n \nabla f^i_k({\bx_i}^{r-1}) 
 	\end{align}
 	and note that \eqref{cross_exp_0_new} holds true, we can show that  the cross terms in the step before $(i)$ can all be eliminated. 
 	 In step $(ii)$ of \eqref{bounded_ybar_via_x}, we use the property that  $\mathbb{E}\|w_j-\mathbb{E}(w_j)\|^2\le \mathbb{E}\|w_j\|^2$ and $\mathbb{E}\|w_j\|^2=\mathbb{E}\|w_k\|^2$ with $w_j:= \frac{1}{m} \sum_{i=1}^m    \nabla f_j^i(\bx_i^{r}) -  \frac{1}{m} \sum_{i=1}^m    \nabla f_j^i(\bx_i^{r-1})$; in $(iii)$ we use Jensen's inequality, and the last inequality $(iv)$  follows Assumption \ref{A1}.

 	Therefore, by combining \eqref{bounded_ybar_per_iterates} and \eqref{bounded_ybar_via_x},  we have for all $(n_r - 1)q +1 \le r \le n_rq - 1$,
 	\begin{align}\label{lemma1_per_iterates}
 	\mathbb{E}\left\|\bar{\by}^r -  \frac{1}{m} \sum_{i=1}^m \nabla f^i({\bx}_i^r)\right\|^2   \le  \mathbb{E} \left\|\bar{\by}^{r-1} -\frac{1}{m} \sum_{i=1}^m  \nabla f^i({\bx_i}^{r-1})\right\|^2  + \frac{L^2}{m|S_2|}      \mathbb{E} \left\|  \bx^{r} -\bx^{r-1} \right\|^2.
 	\end{align}
 	
 Next, note that we have the following bound on $\mathbb{E} \|  \bx^{r} -\bx^{r-1} \|^2$ for all $r\ge 1$:
 	 	\begin{align}
 	 	\mathbb{E} \|\bx^{r}-\bx^{r-1}\|^2 &\stackrel{\eqref{update:x}}{=} \mathbb{E} \|   \bW\bx^{r-1} -\alpha \by^{r-1} - \bx^{r-1} \|^2\nonumber\\
 	 	&\stackrel{(i)} \le 2 \mathbb{E} \|   \bW\bx^{r-1} - \bx^{r-1} \|^2 +2 \alpha^2  \mathbb{E}\| \by^{r-1}\|^2 \nonumber\\
 	 	&\stackrel{(ii)} \le  2  \mathbb{E}\|   ( \bW - \bI) ( \bx^{r-1} -  \bOne \bar{\bx}^{r-1} ) \|^2  +2\alpha^2  \mathbb{E} \| \by^{r-1}\|^2 \nonumber\\
 	 	&\stackrel{(iii)} \le  8  \mathbb{E} \|   \bx^{r-1} - \bOne \bar{\bx}^{r-1} \|^2 +4 \alpha^2  \mathbb{E}\| \by^{r-1} - \bOne \bar{\by}^{r-1}\|^2+4 \alpha^2  \mathbb{E}\|  \bOne \bar{\by}^{r-1}\|^2, \; \forall~r\ge 1 \label{eq:diff:x}
 	 	\end{align}
 	 	where in $(i)$ we apply the Cauchy-Schwarz inequality,  $(ii)$ follows that $\bW \bOne=\bOne$ from Assumption \ref{A3}, and $(iii)$ applies the fact $\|\bW-\bI\|\le \|\bW\|+\|\bI\|\le 2$ (due to Assumption \ref{A3} and the Cauchy-Schwarz inequality).

 	Telescoping the above inequality 	\eqref{lemma1_per_iterates} 
 	 over the $n_r$-th inner loop, that is from $(n_r - 1)q + 1$ to $r$, we obtain the following series of inequalities %\mhcomment{remove 1/n sum}
 	\begin{align*}
 	& \mathbb{E}\|\bar{\by}^r -  \frac{1}{m} \sum_{i=1}^m \nabla f^i({\bx}_i^r)\|^2 \\
 	\le &  \frac{L^2}{m|S_2|}  \sum_{t=(n_r-1)q+1}^{r}      \mathbb{E} \|  \bx^{t} -\bx^{t-1} \|^2
 	+  \mathbb{E} \|\bar{\by}^{(n_r-1)q} -\frac{1}{m} \sum_{i=1}^m      \nabla f^i({\bx_i}^{(n_r-1)q})\|^2\\
 	\stackrel{\eqref{eq:diff:x}}  \le & \frac{8 L^2}{m|S_2|}    \sum_{t=(n_r-1)q+1}^{r}   \mathbb{E}\|  {\bx}^{t-1}  - \bOne \bar{\bx}^{t-1}\|^2 +\frac{4 \alpha^2 L^2 }{m|S_2|}   \sum_{t=(n_r-1)q+1}^{r}  \mathbb{E}  \|     \by^{t-1} - \bOne \bar{\by}^{t-1} \|^2\\
 	&+\frac{4\alpha^2 L^2 }{m|S_2|}   \sum_{t=(n_r-1)q+1}^{r}   \mathbb{E} \|     \bOne \bar{\by}^{t-1} \|^2+\mathbb{E} \|\bar{\by}^{(n_r-1)q} -\frac{1}{m} \sum_{i=1}^m     \nabla f^i({\bx_i}^{(n_r-1)q})\|^2\\
 	\stackrel{(i)} \le & \frac{8 L^2}{m|S_2|}   \sum_{t=(n_r-1)q}^{r}  \mathbb{E}\|  {\bx}^{t}  - \bOne \bar{\bx}^{t}\|^2 +\frac{4 \alpha^2 L^2 }{m|S_2|}   \sum_{t=(n_r-1)q}^{r} \mathbb{E} \|     \by^{t} - \bOne \bar{\by}^{t} \|^2 \\
 	&+\frac{4 \alpha^2 L^2 }{|S_2| }    \sum_{t=(n_r-1)q}^{r} \mathbb{E}  \|     \bar{\by}^{t} \|^2+\mathbb{E} \|\bar{\by}^{(n_r-1)q} -\frac{1}{m} \sum_{i=1}^m   \nabla f^i({\bx_i}^{(n_r-1)q})\|^2,
 	\end{align*}
 	where 
 	 in $(i)$ we change the index in the summation, and add three non-negative terms (one for each sum). This concludes the first part of this lemma.

 	Next we show that  \eqref{lemma1-2} holds true. First, by using the same argument as in \eqref{C_Expectation}, we can obtain the following
 	\begin{align*}
 	\mathbb{E} [ {\bv}^{r}- {\bv}^{r-1} | \mathcal{F}_{r}] \stackrel{\eqref{update:v}}=  \mathbb{E} \left[ \frac{1}{|S_2|}  \sum_{j\in S_2} \left[ \nabla f_j(\bx^{r}) - \nabla f_j(\bx^{r-1})  \right]   \bigg| \mathcal{F}_{r}\right] =     \nabla f (\bx^{r}) - \nabla f (\bx^{r-1}).
 	\end{align*}
 	By using the above fact, and that conditioning on $\mathcal{F}_r$, $\bx^r, \bx^{r-1}$ and $\bv^{r-1}$, we obtain the following:
 	\begin{align}\label{eq:v:unbiased}
 	\mathbb{E} [\langle {\bv}^{r-1} -  \nabla f({\bx}^{r-1}), {\bv}^{r}-{\bv}^{r-1} -  \nabla f({\bx}^r)+  \nabla f({\bx}^{r-1}) \rangle|\mathcal{F}_r] = 0.
 	\end{align}
 	Then it is straightforward to obtain following:
 	\begin{align*}
 	\mathbb{E}\| {\bv}^r -    \nabla f({\bx}^r)\|^2	
 &	 =    \mathbb{E} \| {\bv}^{r-1} -  \nabla f({\bx}^{r-1}) +\bv^r-\bv^{r-1}-  \nabla f({\bx}^r)+  \nabla f({\bx}^{r-1})\|^2\\
 & \stackrel{\eqref{eq:v:unbiased}}=    \mathbb{E} \| {\bv}^{r-1} -  \nabla f({\bx}^{r-1})\|^2  + \mathbb{E}\|\bv^r-\bv^{r-1} -  \nabla f({\bx}^r)+  \nabla f({\bx}^{r-1})\|^2\\
 & \stackrel{\eqref{update:v}}=    \mathbb{E} \| {\bv}^{r-1} -  \nabla f({\bx}^{r-1})\|^2  +\mathbb{E} \|\frac{1}{ |S_2|}   \sum_{j\in S_2}   \nabla f_j(\bx^{r})   -  \frac{1}{ |S_2|}  \sum_{j\in S_2}  \nabla f_j(\bx^{r-1}) -  \nabla f({\bx}^r)+  \nabla f({\bx}^{r-1})\|^2\\
 	&\stackrel{(i)} \le  \mathbb{E} \| {\bv}^{r-1} -  \nabla f({\bx}^{r-1})\|^2  + \frac{1}{|S_2|} \mathbb{E}   \|  \nabla f_j(\bx^{r})   -   \nabla f_j(\bx^{r-1})\|^2\\
 	&\stackrel{(ii)} \le  \mathbb{E} \| {\bv}^{r-1} -  \nabla f({\bx}^{r-1})\|^2  + \frac{L^2}{|S_2|} \mathbb{E}   \|  \bx^{r}  - \bx^{r-1} \|^2,
 	\end{align*}
 	where (i) and (ii) follow similar arguments as in \eqref{bounded_ybar_via_x}.
 	 
 	Telescoping the above inequality over $r$ from $(n_r - 1)q + 1$ to $r$, we obtain that
 	\begin{align*}
 	\mathbb{E}\| {\bv}^r -    \nabla f({\bx}^r)\|^2	
 	&\le  \mathbb{E} \| {\bv}^{(n_r - 1)q} -  \nabla f({\bx}^{(n_r - 1)q})\|^2  + \frac{L^2}{|S_2|} \sum_{t=(n_r - 1)q+1}^{r} \mathbb{E}   \|  \bx^{t}  - \bx^{t-1} \|^2\\
 	&\le  \mathbb{E} \| {\bv}^{(n_r - 1)q} -  \nabla f({\bx}^{(n_r - 1)q})\|^2  + \frac{L^2}{|S_2|} \sum_{t=(n_r - 1)q}^{r} \mathbb{E}   \|  \bx^{t+1}  - \bx^{t} \|^2.
 	\end{align*}
 	This completes the proof of the second part of this lemma. 
 \end{proof}
 
 \section{Proof of \leref{Lemma-Descent}}
 \begin{proof}
 	We first establish the relation of function values between the iterates. According to the gradient Lipschitz continuity Assumption \ref{A1}, we have
 	\begin{align*}
 	f(\bar{\bx}^{r+1})&\le f(\bar{\bx}^r)+\langle\nabla f(\bar{\bx}^r),\bar{\bx}^{r+1}-\bar{\bx}^r\rangle+\frac{L}{2}\|\bar{\bx}^{r+1}-\bar{\bx}^r\|^2\\
 	&	\stackrel{(i)}= f(\bar{\bx}^r)-\alpha \langle \nabla f(\bar{\bx}^r), \bar{\by}^r \rangle + \frac{\alpha^2 L}{2} \|\bar{\by}^r\|^2\\
 	&	\stackrel{(ii)}= f(\bar{\bx}^r)-\alpha \langle \nabla f(\bar{\bx}^r)  -\bar{\by}^r, \bar{\by}^r \rangle -\alpha \|  \bar{\by}^r \|^2 + \frac{\alpha^2 L}{2} \|\bar{\by}^r\|^2\\
 	&	\stackrel{(iii)}\le f(\bar{\bx}^r)+\frac{\alpha}{2}\| \nabla f(\bar{\bx}^r)  -\bar{\by}^r\|^2- \frac{\alpha}{2} \|\bar{\by}^r \|^2   + \frac{\alpha^2 L}{2} \|\bar{\by}^r\|^2\\
 	&	\stackrel{(iv)}\le f(\bar{\bx}^r) - \left(\frac{\alpha}{2}   - \frac{\alpha^2 L}{2} \right)\|\bar{\by}^r\|^2 +{\alpha}\| \nabla f(\bar{\bx}^r) -\frac{1}{m} \sum_{i=1}^m \nabla f^i(\bx_i^r)\|^2 +{\alpha}\| \frac{1}{m} \sum_{i=1}^m \nabla f^i(\bx_i^r) -\bar{\by}^r\|^2,
 	\end{align*}
 	where we simply plug in the iterates \eqref{x-update} in $(i)$, add and subtract a term $\bar{\by}^r$ in $(ii)$, and apply the Cauchy-Schwarz inequality in $(iii)$ and $(iv)$.

 	Then the third term can be further quantified as below,
 	\begin{align*}
 	&\| \nabla f(\bar{\bx}^r) -\frac{1}{m} \sum_{i=1}^m \nabla f^i(\bx_i^r)\|^2 \stackrel{(i)}\le \frac{1}{m} \sum_{i=1}^m  \| \nabla f^i(\bar{\bx}^r) - \nabla f^i(\bx_i^r)\|^2 \stackrel{(ii)} \le \frac{1}{m}\sum_{i=1}^m  L^2 \| {\bx_i^r} - \bar{\bx}_i^r \|^2 = \frac{L^2}{m}   \| {\bx^r} - \bOne \bar{\bx}^r \|^2
 	\end{align*}
 	where in $(i)$ we use the Jensen's inequality and in $(ii)$ we use the Lipschitz Assumption \ref{A1}.

 	Taking expectation on both sides and combine with \eqref{lemma1-1} in Lemma \ref{bounded_variance}, we have 
 	\begin{align*}
 	\mathbb{E}[f(\bar{\bx}^{r+1})] &\le \mathbb{E} [f(\bar{\bx}^r)] - \left(\frac{\alpha}{2}   - \frac{\alpha^2 L}{2} \right)\mathbb{E}\|\bar{\by}^r\|^2 + \frac{\alpha L^2}{m}  \mathbb{E} \| {\bx^r} - \bOne \bar{\bx}^r \|^2 +  \frac{8 \alpha L^2}{m|S_2|}   \sum_{t=(n_r-1)q}^{r}  \mathbb{E}\|  {\bx}^{t}  - \bOne \bar{\bx}^{t}\|^2 \\
 	&+\frac{4 \alpha^3 L^2 }{m|S_2|}   \sum_{t=(n_r-1)q}^{r} \mathbb{E} \|     \by^{t} - \bOne \bar{\by}^{t} \|^2+\frac{4 \alpha^3 L^2 }{|S_2| }    \sum_{t=(n_r-1)q}^{r} \mathbb{E}  \|     \bar{\by}^{t} \|^2 + \alpha \mathbb{E} \|\bar{\by}^{(n_r-1)q} -\frac{1}{m} \sum_{t=1}^m    \nabla f^t({\bx_t}^{(n_r-1)q})\|^2\\
 	&\le \mathbb{E} [f(\bar{\bx}^r)] - \left(\frac{\alpha}{2}   - \frac{\alpha^2 L}{2} \right)\mathbb{E}\|\bar{\by}^r\|^2 + \frac{\alpha L^2}{m}  \mathbb{E} \| {\bx^r} - \bOne \bar{\bx}^r \|^2 +  \frac{8 \alpha L^2}{m|S_2|}   \sum_{t=(n_r-1)q}^{r}  \mathbb{E}\|  {\bx}^{t}  - \bOne \bar{\bx}^{t}\|^2 \\
 	&+\frac{4 \alpha^3 L^2 }{m|S_2|}   \sum_{t=(n_r-1)q}^{r} \mathbb{E} \|     \by^{t} - \bOne \bar{\by}^{t} \|^2+\frac{4 \alpha^3 L^2 }{|S_2| }    \sum_{t=(n_r-1)q}^{r} \mathbb{E}  \|     \bar{\by}^{t} \|^2 + \alpha \epsilon_1,
 	\end{align*}
 	where the last inequality we use the definition of $\epsilon_1$ in \eqref{epsilon_1}.

 	Next, telescoping  over one inner loop, that is $r$ from $(n_r - 1)q$ to $r$, we have 
 	\begin{align*}
 	&\mathbb{E}[f(\bar{\bx}^{r+1})] \le \mathbb{E}[f(\bar{\bx}^{(n_r-1)q})] - \left(\frac{\alpha}{2}   - \frac{\alpha^2 L}{2} \right) \sum_{t=(n_r-1)q}^{r}\mathbb{E}\|\bar{\by}^t\|^2 + \frac{\alpha L^2}{m}  \sum_{t=(n_r-1)q}^{r} \mathbb{E}\| {\bx^t} - \bOne \bar{\bx}^t \|^2 \\
 	&+  \frac{8 \alpha L^2}{m|S_2|}  \sum_{t=(n_r-1)q}^{r} \sum_{k=(n_r-1)q}^{t} \mathbb{E} \|  {\bx}^{k}  - \bOne \bar{\bx}^{k}\|^2  +\frac{4 \alpha^3 L^2 }{m|S_2|}   \sum_{t=(n_r-1)q}^{r} \sum_{k=(n_r-1)q}^{t}   \mathbb{E}\|     \by^{k} - \bOne \bar{\by}^{k}\|^2\\
 	&+ \frac{4 \alpha^3 L^2}{|S_2|}  \sum_{t=(n_r-1)q}^{r} \sum_{k=(n_r-1)q}^{t}  \mathbb{E} \|     \bar{\by}^{k} \|^2 + \alpha \sum_{t=(n_r-1)q}^{r} \epsilon_1\\
 	& \stackrel{(i)}\le \mathbb{E} [f(\bar{\bx}^{(n_r-1)q})] - \left(\frac{\alpha}{2}   - \frac{\alpha^2 L}{2} - \frac{4 \alpha^3 L^2 q}{|S_2|} \right) \sum_{t=(n_r-1)q}^{r}\mathbb{E}\|\bar{\by}^t\|^2 + \frac{\alpha L^2}{m}  \sum_{t=(n_r-1)q}^{r} \mathbb{E}\| {\bx^t} - \bOne \bar{\bx}^t \|^2 \\
 	&+  \frac{8 \alpha L^2q}{m|S_2|}  \sum_{t=(n_r-1)q}^{r}  \mathbb{E} \|  {\bx}^{t}  - \bOne \bar{\bx}^{t}\|^2  +\frac{4 \alpha^3 L^2q }{m|S_2|}   \sum_{t=(n_r-1)q}^{r}  \mathbb{E}  \|     \by^{t} - \bOne \bar{\by}^{t}\|^2+ \alpha \sum_{t=(n_r-1)q}^{r} \epsilon_1,
 	\end{align*}
 	where ${(i)}$ follows  the fact that, for any sequence $\{a^i\}$, and an index  $r\le n_rq-1$, we have
 	\begin{align}\label{qsum}
 	\sum_{t=(n_r-1)q}^{r} \sum_{k=(n_r-1)q}^{t} a^k \le \sum_{t=(n_r-1)q}^{r} \sum_{k=(n_r-1)q}^{r} a^k \le q \sum_{k=(n_r-1)q}^{r} a^k.
 	\end{align}

 	Then utilizing the fact that  
 	\begin{align}\label{sum_over_outer_loop}
 	\mathbb{E} [f(\bar{\bx}^{r+1})] -\mathbb{E} [f(\bar{\bx}^0)] = \mathbb{E} [f(\bar{\bx}^{r+1})]        -  \mathbb{E}[f(\bar{\bx}^{(n_r-1)q})] + \cdots +  \mathbb{E}[f(\bar{\bx}^{2q})]   -  \mathbb{E}[f(\bar{\bx}^q)] +  \mathbb{E}[f(\bar{\bx}^q)]        -  \mathbb{E} [ f(\bar{\bx}^0)],
 	\end{align}
 	we have  
 	\begin{align*}
 	\mathbb{E} [f(\bar{\bx}^{r+1})] &\le  \mathbb{E} [f(\bar{\bx}^0)] - \left(\frac{\alpha}{2}   - \frac{\alpha^2 L}{2} - \frac{4 \alpha^3 L^2 q}{|S_2|} \right) \sum_{t=0}^{r}\mathbb{E}\|\bar{\by}^t\|^2 + \frac{\alpha L^2}{m}  \sum_{t=0}^{r}\mathbb{E} \| {\bx^t} - \bOne \bar{\bx}^t \|^2 \\
 	&+  \frac{8 \alpha L^2q}{m|S_2|}  \sum_{t=0}^{r}  \mathbb{E} \|  {\bx}^{t}  - \bOne \bar{\bx}^{t}\|^2  +\frac{4 \alpha^3 L^2q }{m|S_2|}   \sum_{t=0}^{r}   \mathbb{E} \|     \by^{t} - \bOne \bar{\by}^{t}\|^2 + \alpha (r+1) \epsilon_1,
 	\end{align*}
 	which completes the proof.
 \end{proof}

 \section{Proof of \leref{Lemma-Contraction}}
 \begin{proof}
 	First, using the Assumption \ref{A3} on $\bW$, we can obtain the contraction property of the iterates, i.e.,
 	\begin{align}\label{W_contraction}
 	\|\bW\bx^r-\bOne\bar{\bx}^r\|=\|\bW(\bx^r-\bOne\bar{\bx}^r)\|\le\eta\|\bx^r-\bOne\bar{\bx}^r\|.
 	\end{align}
 	To see why the inequality holds true, note that $\bOne^{\T}(\bx^r-\bOne\bar{\bx}^r)=0$, that is, $ \bx^r-\bOne\bar{\bx}^r$ is orthogonal $\bOne$, which is the eigenvector corresponding to the largest eigenvalue of $\bW$. Combining with the fact that  $|\ubar{\lambda}_{\max} (\bW)|=\eta <1$, we obtain the above inequality.
 	
 	Then applying the definition of $\bx$ iterates \eqref{update:x} and the Cauchy-Schwartz inequality, we have
 	\begin{align*}
 	\|\bx^{r+1}-\bOne\bar{\bx}^{r+1}\|^2\stackrel{\eqref{update:x}}=& \|\bW \bx^r-\alpha \by^r-\bOne(\bar{\bx}^r-\alpha \bar{\by}^r)\|^2  \\
 	\le &(1+\beta)\|\bW \bx^r -\bOne\bar{\bx}^r\|^2+\left(1+\frac{1}{\beta}\right)\alpha^2 \|\by^r-\bOne\bar{\by}^r\|^2  \\
 	\stackrel{\eqref{W_contraction}} \le& (1+\beta)\eta^2 \| \bx^r -\bOne\bar{\bx}^r\|^2+\left(1+\frac{1}{\beta}\right)\alpha^2 \|\by^r-\bOne\bar{\by}^r\|^2,
 	\end{align*}
 	where $\beta$ is some constant parameter to be tuned later. Then, taking expectation on both sides we are able to obtain \eqref{x_contraction}.
 	
 	Similarly, we have
 	\begin{align*}
 	\|\by^{r+1}-\bOne \bar{\by}^{r+1}\|^2 &\stackrel{\eqref{update:y}}= \|\bW\by^{r} +\bv^{r+1}-\bv^r - \bOne  (  \bar{\by}^r+\bar{\bv}^{r+1} -  \bar{\bv}^{r})\|^2\\
 	&\le (1+\beta)   \|\bW\by^{r} - \bOne  \bar{\by}^r\|^2 +\left(1+\frac{1}{\beta}\right)  \left\| \left(\bI -  \frac{\bOne \bOne^T}{m}\right)  (   {\bv}^{r+1} -  {\bv}^{r}) \right\|^2\\
 	&\le (1+\beta) \eta^2 \|\by^{r} - \bOne  \bar{\by}^r\|^2 +\left(1+\frac{1}{\beta}\right)  \|{\bv}^{r+1} -  {\bv}^{r} \|^2,
 	\end{align*}
 	where in the last inequality we also use 
 $\|\bI-\frac{1}{m}\bOne\bOne^{\T}\|<1$.
 	
 	After taking expectation on both sides and combining the following inequalities,  the proof for \eqref{y_contraction} is complete.

 	To further bound the term $\|\bv^{r+1}-\bv^r\|^2$, consider that we have  $(n_r - 1)q   \le r \le n_rq - 1$, that is $r$ is taken within one inner loop. We will divide the analysis into two cases. 
 	
 	{\bf Case 1)} For all   $(n_r - 1)q   \le r \le n_rq -2 $, we have  mod$(r+1, q)\ne 0$ and the following is straightforward:
 	\begin{align}\label{vv_ne_zero}
 	\mathbb{E}\|\bv^{r+1}-\bv^{r}\|^2 &\stackrel{\eqref{update:v}}= \mathbb{E} \|\frac{1}{|S_2|} \sum_{j\in S_2} \left[ \nabla f_j(\bx^{r+1}) - \nabla f_j(\bx^{r})  \right]  \|^2 \nonumber\\
 	&\stackrel{(i)}\le \frac{1}{|S_2|}\mathbb{E} \sum_{j\in S_2}\|\nabla f_j(\bx^{r+1}) - \nabla f_j(\bx^{r})   \|^2 \nonumber\\
 	&=\frac{1}{n}\sum_{j=1}^{n}\|\nabla f_j(\bx^{r+1}) - \nabla f_j(\bx^{r})   \|^2 \nonumber\\
 	&\stackrel{(ii)}\le L^2 \mathbb{E} \|\bx^{r+1}-\bx^r\|^2,
 	\end{align}
 	where in $(i)$ we use Jensen's inequality and in $(ii)$ we use Assumption \ref{A1}.
 	
 	{\bf Case 2)} If $r=n_rq-1$, we have mod$(r+1, q) = 0$. 
 	Therefore,
 	\begin{align}\label{vv_e_zero}
 	\mathbb{E}\|\bv^{r+1}-\bv^{r}\|^2 &= \mathbb{E} \|\bv^{r+1}-\nabla f(\bx^{r+1})+\nabla f(\bx^{r+1})-\nabla f(\bx^{r})+\nabla f(\bx^{r}) -\bv^r \|^2 \nonumber \\
 	&\stackrel{(i)} \le 3\mathbb{E}\|\bv^{r+1}-\nabla f(\bx^{r+1})\|^2+3\mathbb{E} \|    \nabla f(\bx^{r+1}) -\nabla f(\bx^{r})  \|^2 +3\mathbb{E} \|    \nabla f(\bx^{r}) -\bv^r  \|^2 \nonumber \\
 	&\stackrel{(ii)}\le 3\epsilon_2 + 3L^2  \mathbb{E} \|\bx^{r+1}-\bx^r\|^2 + 3\sum_{t=(n_r-1)q}^r\frac{L^2}{|S_2|} \mathbb{E} \|\bx^{t+1}-\bx^t\|^2 + 3\epsilon_2,
 	\end{align}
 	where in $(i)$ we use the Cauchy-Schwarz inequality; in $(ii)$ we apply \eqref{lemma1-2} from  \leref{bounded_variance}, Assumption \ref{A1}, and  $\mathbb{E}\|\bv^{r}-\nabla f(\bx^{r})\|^2 \le \epsilon_2$ for all mod$(r, q)=0$.

 	Next,  telescoping $\|\bv^{r+1}-\bv^{r}\|^2$ over $r$ from $(n_r - 1)q$ to $r$. Since $r\le n_rq-1$,  we have at most one follows \eqref{vv_e_zero} and all the rest follow \eqref{vv_ne_zero}. Therefore, we obtain  
 	\begin{align*}
 	\sum_{t=(n_r-1)q}^r \mathbb{E}\|\bv^{t+1}-\bv^{t}\|^2 &\le    \sum_{t=(n_r-1)q}^r L^2  \mathbb{E} \|\bx^{t+1}-\bx^t\|^2 +6\epsilon_2+  
 	2L^2  \mathbb{E} \|\bx^{r+1}-\bx^r\|^2+\sum_{t=(n_r-1)q}^r  \frac{3L^2}{|S_2|} \mathbb{E} \|\bx^{t+1}-\bx^t\|^2\\
 	&\le    \sum_{t=(n_r-1)q}^r 6L^2  \mathbb{E} \|\bx^{t+1}-\bx^t\|^2 +6\epsilon_2.
 	\end{align*}
 	Through a similar step as \eqref{sum_over_outer_loop}, the following is obvious
 	\begin{align*}
 	\sum_{t=0}^r \mathbb{E}\|\bv^{t+1}-\bv^{t}\|^2  &\le 6(r+1) \epsilon_2+ \sum_{t=0}^r 6L^2 \mathbb{E} \|\bx^{t+1}-\bx^t\|^2.
 	\end{align*}

 By combining \eqref{eq:diff:x}, i.e., 
 \begin{align*}
 \mathbb{E} \|\bx^{r+1}-\bx^r\|^2 &
\le  8  \mathbb{E} \|   \bx^{r} - \bOne \bar{\bx}^{r} \|^2 +4 \alpha^2  \mathbb{E}\| \by^r - \bOne \bar{\by}^r\|^2+4 \alpha^2  \mathbb{E}\|  \bOne \bar{\by}^r\|^2, \; \forall~r\ge 0,
 \end{align*}
we  complete the proof.  \end{proof}

\section{Proof of \leref{Lemma-Potential}}
 \begin{proof}
 	We first introduce an intermediate function $\textsl{P}(\bx^r)$ to facilitate the analysis,
 	\begin{align*}
 	\textsl{P}(\bx^r) & :=  \mathbb{E} \|\bx^r-\bOne\bar{\bx}^r\|^2+   \alpha   \mathbb{E}  \|{\by}^r-\bOne{\bar{\by}}^r\|^2.
 	\end{align*}
 	Obviously, we have $\textsl{H}(\bx^r)=\mathbb{E}[f(\bar{\bx}^r)]+{\frac{1}{m}} \textsl{P}(\bx^r)$.
 	
 	By applying \eqref{x_contraction} and \eqref{y_contraction} in \leref{Lemma-Contraction} we have
 	\begin{align*}
 	\textsl{P}(\bx^{r+1}) - \textsl{P}(\bx^r) &
 	\le (1+\beta)\eta^2  \mathbb{E} \| \bx^r -\bOne\bar{\bx}^r\|^2+(1+\frac{1}{\beta})\alpha^2  \mathbb{E} \|\by^r-\bOne\bar{\by}^r\|^2 + \alpha  (1+\beta) \eta^2  \mathbb{E} \|\by^{r} - \bOne  \bar{\by}^r\|^2 \\
 	&+ \alpha (1+\frac{1}{\beta})   \mathbb{E} \|   \bv^{r+1} - \bv^{r} \|^2 -   \mathbb{E} \|\bx^r-\bOne\bar{\bx}^r\|^2 -   \alpha    \mathbb{E} \|{\by}^r-\bOne{\bar{\by}}^r\|^2\\
 	& = - \left( 1 - (1+\beta)\eta^2 \right)  \mathbb{E}\| \bx^r -\bOne\bar{\bx}^r\|^2 -\left(  \alpha - \alpha  (1+\beta) \eta^2-(1+\frac{1}{\beta})\alpha^2 \right)  \mathbb{E}\|\by^r-\bOne\bar{\by}^r\|^2 \\
 	&+ \alpha (1+\frac{1}{\beta}) \mathbb{E} \|   \bv^{r+1} - \bv^{r} \|^2.
 	\end{align*}
 	
 	Next, summing over the iteration from $0$ to $r$   we obtain
 	\begin{align} \label{P_Descent}
 	\textsl{P}(\bx^{r+1}) -  \textsl{P}(\bx^0) \le &- \left( 1 - (1+\beta)\eta^2  \right)  \sum_{t=0}^{r}\mathbb{E}\| \bx^t -\bOne\bar{\bx}^t\|^2\\
 	&-\left(  \alpha - \alpha  (1+\beta) \eta^2-(1+\frac{1}{\beta})\alpha^2  \right) \sum_{t=0}^{r} \mathbb{E}\|\by^t-\bOne\bar{\by}^t\|^2 \nonumber\\
 	&  + \alpha (1+\frac{1}{\beta})\sum_{t=0}^{r} \mathbb{E}  \|   \bv^{t+1} - \bv^{t} \|^2. \nonumber
 	\end{align}

 	If we further pick $q=|S_2|$, then \leref{Lemma-Descent} becomes
 	\begin{align}  \label{F_Descent_online}
 	\mathbb{E} [f(\bar{\bx}^{r+1}) ]\le & \mathbb{E}[f(\bar{\bx}^0)] - \left(\frac{\alpha}{2}   - \frac{\alpha^2 L}{2} - 4 \alpha^3 L^2  \right) \sum_{t=0}^{r}\mathbb{E}\|\bar{\by}^t\|^2 \\
 	&+ \frac{9 \alpha L^2}{m}  \sum_{t=0}^{r}  \mathbb{E} \|  {\bx}^{t}  - \bOne \bar{\bx}^{t}\|^2  +\frac{4 \alpha^3 L^2 }{m}   \sum_{t=0}^{r}   \mathbb{E} \|     \by^{t} - \bOne \bar{\by}^{t}\|^2  +\alpha (r+1) \epsilon_1.\nonumber
 	\end{align}
 	Similarly, \eqref{v_contraction} of \leref{Lemma-Contraction} becomes  
 	\begin{align}\label{V_Descent_online}
 	\sum_{t=0}^r \mathbb{E}\|\bv^{t+1}-\bv^{t}\|^2  &\le  48L^2 \sum_{t=0}^r \mathbb{E} \|   \bx^{t} - \bOne \bar{\bx}^{t} \|^2 +24L^2\alpha^2 \sum_{t=0}^r   \mathbb{E}\| \by^t - \bOne \bar{\by}^t\|^2+24L^2 \alpha^2 \sum_{t=0}^r  \mathbb{E}\|  \bOne \bar{\by}^t\|^2+6(r+1) \epsilon_2.
 	\end{align}

 	Therefore,  combine \eqref{P_Descent}, \eqref{F_Descent_online} and \eqref{V_Descent_online}, we have
 	\begin{align*}
 	&\textsl{H}(\bx^{r+1}) - \textsl{H}(\bx^0) \le- \left(\frac{\alpha}{2}   - \frac{\alpha^2 L}{2} - {4 \alpha^3 L^2 } -24   (1+\frac{1}{\beta}) \alpha^3L^2\right) \sum_{t=0}^{r}\mathbb{E}\|\bar{\by}^t\|^2 \\
 	& - \left( 1 - (1+\beta)\eta^2 - 48 \alpha (1+\frac{1}{\beta}) L^2 - {9 \alpha L^2}\right){\frac{1}{m}}  \sum_{t=0}^{r}\mathbb{E}\| \bx^t -\bOne\bar{\bx}^t\|^2\\
 	&-\left(  \alpha - \alpha  (1+\beta) \eta^2-(1+\frac{1}{\beta})\alpha^2 -24  (1+\frac{1}{\beta}) \alpha^3L^2 -{4 \alpha^3 L^2 }  \right) {\frac{1}{m}} \sum_{t=0}^{r}\mathbb{E} \|\by^t-\bOne\bar{\by}^t\|^2 +
 	\alpha (r+1) (\epsilon_1+6{\frac{1}{m}}(1+\frac{1}{\beta}) \epsilon_2).
 	\end{align*}
 	This completes the proof of the result. 
 \end{proof}

 \section{Proof of Theorem \ref{th:main}}
 \begin{proof}
 	To begin with, we notice that by applying the update rule from Algorithm \ref{Algorithm_Finite}, then for all mod$(r, q) = 0$,  the following holds true
 	\begin{align} 
 	&\mathbb{E}\|\bv^{r}-\nabla f(\bx^{r})\|^2 \stackrel{\eqref{update:v}}=0,\\
 	&\mathbb{E} \|\bar{\by}^{r} -\frac{1}{mn} \sum_{i=1}^m    \sum_{k=1}^n \nabla f^i_k(\bx_i^r)\|^2  \stackrel{\eqref{y_v_equal_finite}}= 0,
 	\end{align}  
 	which implies $\epsilon_1=\epsilon_2=0$ for \leref{Lemma-Descent}, \leref{Lemma-Contraction}, and \leref{Lemma-Potential}.

 	Next, if we further pick $\beta$ such that $1-(1+\beta)\eta^2>0$ and  choose $0< \alpha < \min\{K_1, K_2, K_3\}$, we can rewrite \leref{Lemma-Potential} as below with   $C_1> 0, C_2>0, C_3>0$,
 \begin{align}\label{lemma4_finite}
 & \textsl{H}(\bx^{r+1}) -  \textsl{H}(\bx^0) \le- C_1  \sum_{t=0}^{r}\mathbb{E}\|\bar{\by}^t\|^2  - C_2  \sum_{t=0}^{r} {\frac{1}{m}} \mathbb{E}\| \bx^t -\bOne\bar{\bx}^t\|^2 -C_3 \sum_{t=0}^{r} {\frac{1}{m}} \mathbb{E}\|\by^t-\bOne\bar{\by}^t\|^2.
 \end{align}
 
 	Therefore the upper bound of the optimality gap can be quantified as following 
 	\begin{align*}
 	&\min_{r\in [T]} \mathbb{E}\|\frac{1}{m} \sum_{i=1}^m \nabla f^i({\bx}_i^r)\|^2 +  {\frac{1}{m}}\mathbb{E}\|\bx^r-\bOne\bar{\bx}^r\|^2 \\
 	\stackrel{(i)}{\le}& \frac{1}{T}\sum_{t=0}^T\mathbb{E}\|\frac{1}{m} \sum_{i=1}^m \nabla f^i({\bx}_i^t)\|^2 +  \frac{1}{T} \sum_{t=0}^{T} {\frac{1}{m}} \mathbb{E}\| \bx^t -\bOne\bar{\bx}^t\|^2\\
 	\stackrel{(ii)}\le &\frac{2}{T}\sum_{t=0}^T   \mathbb{E} \|\bar{\by}^t\|^2  + \frac{2}{T}\sum_{t=0}^T  \mathbb{E}  \|\bar{\by}^t - \frac{1}{m} \sum_{i=1}^m \nabla f^i({\bx}_i^t)\|^2 + \frac{1}{T} \sum_{t=0}^{T} {\frac{1}{m}} \mathbb{E}\| \bx^t -\bOne\bar{\bx}^t\|^2
 	\end{align*}
 	where in $(i)$ we use the definition of expectation over $r$, in $(ii)$ we use the Cauchy-Schwarz inequality.
 	
 	Applying \eqref{lemma1-1} from \leref{bounded_variance} with $\mathbb{E} \|\bar{\by}^{(n_r-1)q} -\frac{1}{m} \sum_{i=1}^m      \nabla f^i({\bx_i}^{(n_r-1)q})\|^2=0$, 
	telescoping over $r$ from $0$ to $T$ (follows similar reasoning as \eqref{qsum} and \eqref{sum_over_outer_loop}), and using the choice of $|S_2|=q=\sqrt{n}$, we have
	\begin{align*}
	\sum_{t=0}^{r}\mathbb{E}\|\bar{\by}^t -  \frac{1}{m} \sum_{i=1}^m \nabla f^i({\bx}_i^t)\|^2 &\le  \frac{8 L^2}{m}   \sum_{t=0}^{r}  \mathbb{E}\|  {\bx}^{t}  - \bOne \bar{\bx}^{t}\|^2 +\frac{4 \alpha^2 L^2}{m}   \sum_{t=0}^{r} \mathbb{E} \|     \by^{t} - \bOne \bar{\by}^{t} \|^2  + 4 \alpha^2 L^2      \sum_{t=0}^{r} \mathbb{E}  \|     \bar{\by}^{t} \|^2.
	\end{align*}
	Combining the above two inequalities we can obtain
 \begin{align}
 		&\min_{r\in [T]} \mathbb{E}\|\frac{1}{m} \sum_{i=1}^m \nabla f^i({\bx}_i^r)\|^2 +  {\frac{1}{m}} \mathbb{E}\|\bx^r-\bOne\bar{\bx}^r\|^2 \\
 	     \le & \left(  \frac{16 L^2}{mT} + \frac{1}{{m}T} \right)  \sum_{t=0}^T  \mathbb{E}\|  {\bx}^{t}  - \bOne \bar{\bx}^{t}\|^2 +\frac{8 \alpha^2 L^2 }{mT}  \sum_{t=0}^T \mathbb{E} \|     \by^{t} - \bOne \bar{\by}^{t} \|^2+\left(\frac{8 \alpha^2 L^2 }{ T}   + \frac{2}{T}\right)\sum_{t=0}^T\|\bar{\by}^t\|^2.
 	\end{align}
 	Further combining with \eqref{lemma4_finite}, we have
 	\begin{align*}
 	\min_{r\in [T]} \mathbb{E}\|\frac{1}{m} \sum_{i=1}^m \nabla f^i({\bx}_i^r)\|^2 + {\frac{1}{m}} \mathbb{E}\|\bx^r-\bOne\bar{\bx}^r\|^2 
 	\le  &  C_0 \cdot \frac{\textsl{H}(\bx^0) -{\textsl{H}(\bx^{T+1})} }{T} \le   C_0 \cdot \frac{\mathbb{E}[f({\bx}^{0})] - \ubar{f}}{T},
 	\end{align*}
where 
\begin{align*}
C_0:= \left(   \frac{8 \alpha^2 L^2+2}{C_1} + \frac{16 L^2+{1}}{m C_2}  +\frac{8 \alpha^2 L^2  }{m C_3}  \right),
\end{align*}
and the last inequality follows from
 	\begin{align*}
 	\textsl{H}(\bx^0) & := \mathbb{E}  [f(\bar{\bx}^{0})]+ \mathbb{E} \|\bx^0-\bOne\bar{\bx}^0\|^2+   \alpha   \mathbb{E} \|{\by}^0-\bOne{\bar{\by}}^0\|^2 =\mathbb{E} [ f(\bar{\bx}^{0})],\\
 	\textsl{H}(\bx^r) & :=  \mathbb{E} [f(\bar{\bx}^{r})]+ \mathbb{E} \|\bx^r-\bOne\bar{\bx}^r\|^2+   \alpha  \mathbb{E}  \|{\by}^r-\bOne{\bar{\by}}^r\|^2 \ge\mathbb{E}  [f(\bar{\bx}^{r})] \ge \ubar{f}.
 	\end{align*}
This completes the proof. \end{proof}

\section{Proof of  \leref{online_lemma5}}

 \begin{proof}

First recall the definition of  $\mathbb{E} [\cdot | \mathcal{F}_r]$ in \leref{bounded_variance}, which is the expectation with respect to the random choice of sample $\xi$, conditioning on $\bx^0, \cdots, \bx^{r}$, $\bv^0, \cdots, \bv^{r-1}$ and $\by^0, \cdots, \by^{r-1}$.

 	Let us define a random variable $u_\xi$ as below and $u_{\ell}$ similarly,
 	\begin{align}
 	u_\xi = \frac{1}{m}\sum_{i=1}^m  \nabla f^i_\xi(\bx_i^{r})-\frac{1}{m} \sum_{i=1}^m     \nabla f^i(\bx_i^r).
 	\end{align}
 	Note that $u_\xi$ and $u_\ell$ are independent random variables conditioning on $\mathcal{F}$. Further,  we have the  following from Assumption \ref{A2}
 	\begin{align}  
 	 \mathbb{E}_\xi\left[   \frac{1}{m}\sum_{i=1}^m  \nabla f^i_\xi(\bx_i^{r})-\frac{1}{m} \sum_{i=1}^m     \nabla f^i(\bx_i^r) \bigg| \mathcal{F}_r\right]=0.
 	 \end{align}
Therefore we have
\begin{align}\label{eliminate_cross}
\mathbb{E}[\langle u_\xi, u_\ell \rangle] = \mathbb{E}_{\mathcal{F}}\mathbb{E}[\langle u_\xi, u_\ell\rangle\mid \mathcal{F}] = \mathbb{E}_{\mathcal{F}}\langle \mathbb{E}[u_\xi\mid \mathcal{F}], \mathbb{E}[u_\ell\mid \mathcal{F}]\rangle =0.
\end{align}

Following the update rule from Algorithm \ref{Algorithm_Online}, we have the following relations 
for all mod$(r, q)=0$
	\begin{align*}
	\mathbb{E} \left\|\bar{\by}^{r} -\frac{1}{m} \sum_{i=1}^m    \nabla f^i(\bx_i^r)\right\|^2 \stackrel{\eqref{y_v_equal_online}}= &\mathbb{E} \left\| \frac{1}{m|S_1|}\sum_{i=1}^m \sum_{\xi\in S_1}  \nabla f^i_\xi(\bx_i^{r})-\frac{1}{m} \sum_{i=1}^m     \nabla f^i(\bx_i^r)\right\|^2\\
	 \stackrel{(i)}= &\frac{1}{|S_1|^2}\mathbb{E} \left\|\sum_{\xi\in S_1}\left( \frac{1}{m}\sum_{i=1}^m   \nabla f^i_\xi(\bx_i^{r})-\frac{1}{m} \sum_{i=1}^m     \nabla f^i(\bx_i^r)\right)\right\|^2\\
	\stackrel{(ii)}= & \frac{1}{|S_1|^2} \mathbb{E} \sum_{\xi\in S_1} \left\| \frac{1}{m}\sum_{i=1}^m  \nabla f^i_\xi(\bx_i^{r})-\frac{1}{m} \sum_{i=1}^m     \nabla f^i(\bx_i^r)\right\|^2\\
	\stackrel{(iii)}= & \frac{1}{|S_1|} \mathbb{E}   \left\| \frac{1}{m}\sum_{i=1}^m  \nabla f^i_\xi(\bx_i^{r})-\frac{1}{m} \sum_{i=1}^m     \nabla f^i(\bx_i^r)\right\|^2\\
	\stackrel{(iv)}\le & \frac{1}{m|S_1|}  \sum_{i=1}^m  \mathbb{E}\left\|  \nabla f^i_\xi(\bx_i^{r})-     \nabla f^i(\bx_i^r)\right\|^2\\
	\le & \frac{\sigma^2}{|S_1|},
	\end{align*}
	where in $(i)$ we take out the constant $|S_1|$; in $(ii)$ we eliminate the cross terms via \eqref{eliminate_cross}; in $(iii)$ we use the fact that the following term
	\begin{align}  
 	 \mathbb{E}\|   \frac{1}{m}\sum_{i=1}^m  \nabla f^i_\xi(\bx_i^{r})-\frac{1}{m} \sum_{i=1}^m     \nabla f^i(\bx_i^r)\|^2
 	 \end{align}
 	 are equal across different samples $\xi$; in $(iv)$ we use Jensen's inequality, and the last inequality follows the Assumption \ref{A4}.

Similarly, we have
	\begin{align*}
\mathbb{E} \|{\bv}^{r} - \nabla f({\bx}^{r})\|^2 \stackrel{\eqref{update:v:online}}= &\mathbb{E} 
\left\| \frac{1}{|S_1|}  \sum_{\xi\in S_1}  \nabla f_\xi(\bx^{r})-  \nabla f({\bx}^{r})\right\|^2\\
 = & \frac{1}{|S_1|} \mathbb{E} \|   \nabla f_\xi(\bx^{r})-  \nabla f({\bx}^{r})\|^2\\
 \le & \frac{1}{|S_1|} \sum_{i=1}^m \mathbb{E} \|   \nabla f_\xi^i(\bx^{r})-  \nabla f^i({\bx}^{r})\|^2\\
\le & \frac{m\sigma^2}{|S_1|}.
\end{align*}
This completes the proof.
 \end{proof}

\section{Proof of Theorem \ref{th:main2}}

\begin{proof}

Note that it is easy to check that \leref{bounded_variance}, \leref{Lemma-Descent}, \leref{Lemma-Contraction} and \leref{Lemma-Potential}  still hold true. And the quantity $\epsilon_1$ and $\epsilon_2$ can be determined by \leref{online_lemma5}, i.e., $\epsilon_1=\frac{\sigma^2}{|S_1|}$ and $\epsilon_2=\frac{m\sigma^2}{|S_1|}$. Therefore, \leref{Lemma-Potential} can be rewritten as 
below if we follow	Algorithm \ref{Algorithm_Online}, 
		\begin{align}\label{lemma4_online}
	&\textsl{H}(\bx^{r+1}) -  \textsl{H}(\bx^0) \le- C_1 \sum_{t=0}^{r}\mathbb{E}\|\bar{\by}^t\|^2  - C_2  \sum_{t=0}^{r}{\frac{1}{m}}\mathbb{E}\| \bx^t -\bOne\bar{\bx}^t\|^2 -C_3 \sum_{t=0}^{r} {\frac{1}{m}}\mathbb{E}\|\by^t-\bOne\bar{\by}^t\|^2 +\epsilon_3,
	\end{align}
	with $\epsilon_3=\alpha (r+1) (1+6(1+\frac{1}{\beta})) \frac{\sigma^2}{|S_1|}$.
	
	Therefore the upper bound of the optimality gap can be derived in a similar way as Theorem \ref{th:main},
	\begin{align*}
	&\min_{r\in [T]}\mathbb{E}\|\frac{1}{m} \sum_{i=1}^m \nabla f^i({\bx}_i^r)\|^2 +  {\frac{1}{m}}\mathbb{E}\|\bx^r-\bOne\bar{\bx}^r\|^2 \\
	\le &\frac{2}{T}\sum_{t=0}^T   \mathbb{E} \|\bar{\by}^t\|^2  + \frac{2}{T}\sum_{t=0}^T  \mathbb{E}  \|\bar{\by}^t - \frac{1}{m} \sum_{i=1}^m \nabla f^i({\bx}_i^t)\|^2 + \frac{1}{T} \sum_{t=0}^{T}{\frac{1}{m}}\mathbb{E}\| \bx^t -\bOne\bar{\bx}^t\|^2\\
	\le & \left(  \frac{16 L^2}{mT} + \frac{1}{mT} \right)  \sum_{t=0}^T  \mathbb{E}\|  {\bx}^{t}  - \bOne \bar{\bx}^{t}\|^2 +\frac{8 \alpha^2 L^2 }{mT}  \sum_{t=0}^T \mathbb{E} \|     \by^{t} - \bOne \bar{\by}^{t} \|^2+\left(\frac{8 \alpha^2 L^2 }{ T}   + \frac{2}{T}\right)\sum_{t=0}^T\|\bar{\by}^t\|^2 + \frac{2}{T} \sum_{t=0}^T \frac{\sigma^2}{|S_1|}
	\end{align*}
	Further combining \eqref{lemma4_online} we have
	\begin{align*}
	&\min_{r\in [T]}\mathbb{E}\|\frac{1}{m} \sum_{i=1}^m \nabla f^i({\bx}_i^r)\|^2 +  {\frac{1}{m}}\mathbb{E}\|\bx^r-\bOne\bar{\bx}^r\|^2 \\
	    \le & C_0 \left( \frac{\textsl{H}(\bx^0) -\textsl{H}(\bx^{T+1})+\epsilon_3 }{T} \right)  + \frac{2T+2}{T} \frac{\sigma^2}{|S_1|}  \\
	\le & C_0 \cdot \frac{\mathbb{E}[f({\bx}^{0})] - \ubar{f}}{T} + C_0 \cdot \frac{\alpha(T+1)(7+\frac{6}{\beta}) \sigma^2}{T|S_1|} + \frac{2T+2}{T} \frac{\sigma^2}{|S_1|}.
	\end{align*}
	After picking $|S_1| = \frac{4C_0\alpha(7+\frac{6}{\beta})\sigma^2+8\sigma^2}{\epsilon}$, we complete the proof.
	\end{proof}

\end{document}